\documentclass[12pt, reqno]{amsart}
\usepackage[all,tips]{xy}
\usepackage{latexsym, amsfonts, amsmath, amssymb, mathrsfs}

\newtheorem{theorem}{Theorem}[section]
\newtheorem{lemma}[theorem]{Lemma}
\newtheorem{proposition}[theorem]{Proposition}
\newtheorem{corollary}[theorem]{Corollary}

\newtheorem*{Theorem1.1}{Theorem}
\newtheorem*{Corollary1.2}{Corollary}
\newtheorem*{Theorem1.3}{Theorem}
\newtheorem*{Theorem1.4}{Theorem}
\newtheorem*{Theorem1.5}{Theorem}
\newtheorem*{Theorem1.6}{Theorem}

\theoremstyle{definition}

\newcommand{\norm}[1]{\ensuremath{\left \vert \left \vert #1 \right \vert \right \vert}}
\newcommand{\abs}[1]{\ensuremath{\left\vert #1 \right\vert}}
\newcommand{\Cal}[1]{\ensuremath{\mathcal{#1}}}

\usepackage[inner=1.0in,outer=1.0in,bottom=1.0in, top=1.0in]{geometry}

\def\C{{\mathbb C}}
\def\R{{\mathbb R}}
\def\N{{\mathbb N}}
\def\Z{{\mathbb Z}}
\def\Q{{\mathbb Q}}
\def\H{{\mathbb H}}
\def\a{{\frak a}}
    
\def\O_K{{\mathcal{O}_{K}}}
\def\O_F{{\Cal{O}_{F}}}
\def\N_F{{\Cal{N}_{F/\Q}}}

\def\a{{\mathfrak a}}

\def\O_K{{\Cal{O}_{K}}}
\def\O_F{{\Cal{O}_{F}}}
\def\N_F{{\Cal{N}_{F/\Q}}}

\newcommand{\mz}{\ensuremath{\mathbb Z}}

\newcommand{\shortmod}{\ensuremath{\negthickspace \negthickspace \negthickspace \pmod}}

\newcommand{\half}{\ensuremath{ \frac{1}{2}}}

\newcommand{\intR}{\int_{-\infty}^{\infty}}

\newcommand{\thalf}{\tfrac12}
\newcommand{\sumstar}{\sideset{}{^*}\sum}
\newcommand{\leg}[2]{\left(\frac{#1}{#2}\right)}
\newcommand{\e}[2]{e\left(\frac{#1}{#2}\right)}

\DeclareMathOperator{\vol}{vol}
\numberwithin{equation}{section}
\numberwithin{theorem}{section}

\begin{document}

\title[]{Subconvexity and equidistribution of Heegner points in the level aspect}

\author{Sheng-Chi Liu, Riad Masri, and Matthew P. Young}

\address{Department of Mathematics, Mailstop 3368, Texas A\&M University, College Station, TX 77843-3368 }

\email{scliu@math.tamu.edu}
\email{masri@math.tamu.edu}
\email{myoung@math.tamu.edu}
\thanks{M.Y. was supported by the National Science Foundation under agreement No. DMS-0758235.  Any opinions, findings and conclusions or recommendations expressed in this material are those of the authors and do not necessarily reflect the views of the National Science Foundation.}

\begin{abstract} Let $q$ be a prime and $-D < -4$ be an odd fundamental discriminant such that %$(q,D)=1$ and 
$q$ splits in $\Q(\sqrt{-D})$.  
For $f$ a weight zero Hecke-Maass newform of level $q$ and $\Theta_\chi$ the weight one theta series of level $D$ corresponding to an ideal 
class group character $\chi$ of $\Q(\sqrt{-D})$, 
we establish a hybrid
subconvexity bound for %the Rankin-Selberg $L$-function 
$L(f \times \Theta_\chi,s)$ at $s=1/2$ 
when $q \asymp D^{\eta}$ for $0 < \eta < 1$.  
With this circle of ideas, we show that the Heegner points of level $q$ and discriminant $D$ become equidistributed, in a natural sense, as $q, D \rightarrow \infty$ for $q \leq D^{1/20-\varepsilon}$.
Our approach to these problems is connected to estimating the $L^2$-restriction norm of a Maass form of large level $q$ when restricted to the collection of Heegner points.
We furthermore establish bounds for quadratic twists of Hecke-Maass $L$-functions with simultaneously large level and large quadratic twist, and hybrid bounds for quadratic Dirichlet $L$-functions in certain ranges.
\end{abstract}
%\tableofcontents

\maketitle

\section{Introduction and statement of results}  
Let $f$ be an arithmetically normalized Hecke-Maass newform of weight zero and prime level $q$ with spectral parameter $t_f$. 
Let $K=\Q(\sqrt{-D})$ be an imaginary quadratic field of discriminant $-D < -4$, such that $q$ splits in $K$, with ideal class group $\mathrm{CL}_K$ and class number $h(-D)$. 
Given an ideal class group character $\chi \in \widehat{\mathrm{CL}}_K$, let $\Theta_{\chi}$ be the weight one theta series
of level $D$ corresponding to $\chi$. 
%We will establish the following 
%estimate for the first moment.
\begin{theorem}\label{momentthm}  
With notation as above, we have
% Let $q$ be a prime and $-D < -4$ be an odd fundamental discriminant such that 
% $q$ splits in $K=\Q(\sqrt{-D})$. Then for a 
% Hecke-Maass newform $f$ of weight zero for 
% $\Gamma_0(q)$ we have 
%\begin{align*}
\begin{multline}
\label{e8.4}
\sum_{\chi \in \widehat{\mathrm{CL}}_K}L(f \times \Theta_\chi,\tfrac{1}{2}) = 
\frac{3}{\pi} \frac{h(-D)^2}{\sqrt{D}}\frac{q^2}{q^2-1}L(\mathrm{sym}^2f,1) 
\\
+ O_{t_f, \varepsilon}((qD)^{\varepsilon} \min(q D^{7/16}, q^{3/4} D^{1/4} + q^{1/4} D^{1/2})).
%\end{align*}
\end{multline}
\end{theorem}
%\vspace{0.05in}

We have established an analogue of Theorem \ref{momentthm} for holomorphic forms $f$ which will appear in a subsequent paper.

Recall the conductor of $L(f \times \Theta_\chi,s)$ at $s=1/2$ is $Q=(qD)^2$, 
so the convexity bound is 
%\begin{align*}
$L(f \times \Theta_{\chi},\tfrac{1}{2}) \ll_{t_f, \varepsilon} Q^{\frac{1}{4}+\varepsilon}$. 
%\end{align*}
These central values are nonnegative, and dropping all but one term in
Theorem \ref{momentthm} yields a subconvexity bound when $q \asymp D^{\eta}$ for $0 < \eta < 1$.  There are a variety of cases to consider to state the best bound as a function of $\eta$; for simplicity we record what one obtains with the second bound in \eqref{e8.4} which suffices for subconvexity for $0 < \eta < 1$. %(when $\eta = 0$ one should use the first bound in \eqref{e8.4}).
\begin{corollary}\label{mt} 
%Let $q, D$ and $f$ be as in Theorem \ref{momentthm}, and let $Q=(qD)^2$. 
For $\eta:=\log(q)/\log(D)$ satisfying $0 < \eta < 1$, we have 
\begin{equation}
L(f \times \Theta_\chi,\tfrac{1}{2}) \ll_{t_f, \varepsilon} Q^{\frac{1}{4}+\varepsilon}
\Big( Q^{- \frac{1-\eta}{8(1+\eta)}} + Q^{- \frac{\eta}{8(1+\eta)}} \Big).
\end{equation}
\end{corollary}

We note that Theorem \ref{momentthm} also yields the following quantitative 
nonvanishing result.  

\begin{corollary}\label{mt2} For each $\varepsilon > 0$ there is an effective constant $c=c(\varepsilon, t_f) > 0$ 
such that whenever $q \leq c  D^{1/16-\varepsilon}$, we have $L(f \times \Theta_\chi, 1/2) \neq 0$ for 
some $\chi \in \widehat{\mathrm{CL}}_K$. 
\end{corollary}

The subconvexity problem for %Rankin-Selberg 
$L$-functions in the level aspect 
has been studied extensively in recent years. For example, let $f$ and $g$ be Hecke cusp forms (holomorphic or Maass) 
for $GL_2$ of levels 
$M$ and $N$ respectively. If $M$ is fixed and $N$ varies, subconvexity bounds of the form 
$L(f \times g, 1/2) \ll_f N^{\frac{1}{2}-\delta}$
for some absolute $\delta > 0$ were established in various settings by many authors  %Kowalski, Michel, and VanderKam 
\cite{KMV} %, %Michel 
\cite{Michel} %, Harcos and Michel 
\cite{HMi} %, Blomer and Harcos 
%\cite{BlHa} %, and Michel and Venkatesh 
\cite{MV2}. 
On the other hand, it is also of interest to establish a subconvexity bound for $L(f \times g,1/2)$ when \textit{both} $M$ and $N$ are allowed 
to vary. A model result of this type was established by Michel and Ramakrishnan \cite{MR}, who considered the average of 
$L(f \times \Theta_{\chi}, 1/2)$ over holomorphic forms $f$.  Their result also implies subconvexity for these $L$-functions in the wide range $0 < \eta < 1$ 
(with notation as in Corollary \ref{mt}).  They remark that this subconvexity result is intriguing because ``such uniformity seems hard to achieve by purely analytic methods.''
Here we average over $\chi$ (not $f$), though after some transformations we are led to averaging $L(g \times \chi_{D}, 1/2)$ over $g$, 
where $g$ runs over level $q$ Hecke-Maass forms and $\chi_D$ is a quadratic Dirichlet character.    
Feigon and Whitehouse \cite{FW} generalized the work of \cite{MR} to the number field setting, and Nelson \cite{N} has obtained 
results with $\Theta_{\chi}$ replaced by more general holomorphic cusp forms.
Using a different approach (a second moment), Holowinsky and Munshi \cite{HMu} have obtained subconvexity 
when $M \asymp N^{\eta}$ for $\eta > 0$ in some fixed range; their work does not require the central values to be nonnegative.

The $L$-functions in Theorem \ref{momentthm} arise naturally in various arithmetic problems related to the 
equidistribution of Heegner points (see e.g. \cite{Duke} \cite{Z} \cite{Z2} \cite{Michel} \cite{HMi} \cite{MVICM} \cite{MV}).
%This can also be proved using a formula of Waldspurger/Zhang and subconvexity bounds for quadratic twists of Hecke-Maass $L$--functions. 
We will use some ideas involved in the proof of Theorem \ref{momentthm} to prove a ``sparse'' equidistribution theorem for Galois orbits of Heegner points 
in which both $q$ and $D$ are varying; this is different from the more familiar case as in \cite{Michel} \cite{HMi} where for $q$ fixed 
and $D$ varying the sparsity comes from suborbits of the full Galois orbit.
See Section \ref{section:equi} for a description of the problem and Theorem \ref{thm:equi2} for 
the precise result.
For brevity, we state here the following special consequence. 

%Let $q$ be a prime and $-D < -4$ be an odd 
%fundamental discriminant such that $q$ splits in $K=\Q(\sqrt{-D})$. 
Let $H$ be the Hilbert class field of $K$ and $G:=\textrm{Gal}(H/K) \cong \textrm{CL}_K$. 
The set $\Lambda_D(q)$ of Heegner points of discriminant $-D$ on the modular curve 
$X_0(q)$ splits into two simple, transitive $G$-orbits which are permuted by the Fricke involution 
$w_q$ which acts on weight $0$ forms $f$ (invariant under $\Gamma_0(q)$) by $(w_q f)(z) = f(-1/qz)$; for these facts, see pp. 235-236 of \cite{GZ}.
%Let $G < G_K$ be a subgroup and 
For $\tau \in \Lambda_D(q)$, consider the Galois orbit $G \tau=\{\tau^{\sigma}:~\sigma \in G \}$. The (open) modular curve $Y_0(q)=\Gamma_0(q) \backslash \H$ 
can be written as
\begin{align*}
Y_0(q)=\bigcup_{\omega_q \in \Gamma_0(1)/\Gamma_0(q)}\omega_q^{-1}Y_0(1).
\end{align*}
For any given $\omega_q$, we are interested in an asymptotic count for the number 
$$N_{G, D, \omega_q}:=\# (G\tau \cap \omega_q^{-1}Y_0(1))$$ of Heegner points in $G\tau$ which lie in $\omega_q^{-1}Y_0(1)$ 
as $q, D \rightarrow \infty$. Here we have in mind an analogy with counting primes (say) in an arithmetic progression where the choice of $\omega_q$ 
corresponds to the choice of residue class modulo $q$ in which one is counting primes. Note that by the $SL_2(\R)$-invariance of the hyperbolic measure we have 
\begin{align*}
\frac{\textrm{vol}(\omega_q^{-1}Y_0(1))}{\textrm{vol}(Y_0(q))}=\frac{1}{q+1},
\end{align*}
and thus the volume of $\omega_q^{-1}Y_0(1)$ is becoming very small compared to the total volume of $Y_0(q)$ as $q \rightarrow \infty$. 
%(here we used that $[\Gamma_0(1):\Gamma_0(q)]=q+1$ since $q$ is prime).
 
\begin{theorem} \label{thm:equi} We notation as above we have, uniformly in $\omega_q$, 
\begin{align}
N_{G, D, \omega_q} \sim \frac{h(-D)}{q+1}
\end{align}
as $q, D \rightarrow \infty$ with the restriction $q \leq D^{1/20 - \varepsilon}$.  Furthermore, for $D^{1/20 - \varepsilon} \leq q \leq D^{1/4-\varepsilon}$, we have
\begin{equation}
\label{eq:HeegnerUB}
 N_{G, D, \omega_q} \ll q^{1/4} D^{7/16 + \varepsilon}.
\end{equation}
%Morover, if $G \neq G_K$ we have 
% \begin{align*}
% N_{G, D, \omega_q} \sim \frac{\# G}{q+1}
% \end{align*}
% as $q, D \rightarrow \infty$ with the restriction $q  \leq D^{1/48 - \varepsilon}[G_K:G]^{-2/3}$. 
%The estimate is independent of the choice of $\omega_q$. 
\end{theorem}
The upper bound in \eqref{eq:HeegnerUB} shows that the Heegner points cannot cluster too much into one translate $\omega_q^{-1} Y_0(1)$.

%The equidistribution also follows from a formula of 
%Waldspurger/Zhang and subconvexity bounds for quadratic twists of Hecke-Maass $L$--functions. In particular, for $q$ fixed, the asymptotic 
%formula in Theorem \ref{momentthm} is essentially equivalent to equidistribution (see our description of the proof).   

%

%\begin{theorem}
%\label{thm:equi}
%Let $q$ be prime, with $(q,D) = 1$.  For each $\omega_q \in \Gamma_0(1)/\Gamma_0(q)$, let $N_{D, \omega}$ be the number of Heegner points $\tau^{\sigma}$ of discriminant $D$ and level $q$ 
%with $\tau^{\sigma} \in \omega_q (\Gamma_0(1) \backslash \mathbb{H})$.
%Then
%\begin{equation}
%N_{D, \omega_q} \sim \frac{h(-D)}{[\Gamma_0(1):\Gamma_0(q)]},
%\end{equation}
%as $q, D \rightarrow \infty$ with the restriction $q \leq D^{1/20 - \varepsilon}$.  The estimate does not depend on $\omega_q$.
%\end{theorem}

Our work here has some connections with $L^2$ restriction norms of automorphic forms.  
The formula \eqref{af} below relates the average of $L$-values appearing in Theorem \ref{momentthm} to 
the sum over Heegner points of the modulus squared of a level $q$ Maass form.  It is very interesting to understand this behavior as either $q$ and $D$ vary.  If $D$ is fixed and 
$q$ varies then obviously the number of points is fixed and the problem is really about the sup-norm at special points.  As $q$ and $D$ vary together then we are studying a hybrid version.  
One can see some pleasant analogies between the methods used here and in \cite{BKY}, especially their Theorem 1.7 which is a geodesic restriction bound.

The restriction to prime level $q$ is made to simplify some arguments.  Probably with some extra work one can 
show that our results hold for all square-free integers $q > 1$ and odd fundamental discriminants $-D < -4$ such that 
every prime divisor of $q$ splits in $\Q(\sqrt{-D})$.  For simplicity we also do not keep track of the $t_f$-dependence in our estimates, but it is clear from the proof that the dependence is polynomial (and probably of ``respectable'' degree).

When $q$ is fixed, Michel and Venkatesh \cite{MV} 
proved an asymptotic formula with a power savings in $D$ for the first moment in Theorem \ref{momentthm} and gave applications to nonvanishing. 
An analogous result for central derivatives was obtained by Templier \cite{TemplierDuke}. Earlier, Duke, Friedlander, and Iwaniec \cite{DFI} 
obtained an asymptotic formula for the second moment of class group $L$-functions which 
in Theorem \ref{momentthm} corresponds to replacing $f$ by an Eisenstein series.
The novelty in Theorem \ref{momentthm} is that both $q$ and $D$ are allowed to vary in a fairly wide range.
Our proof is inspired in many ways by \cite{MV}; see in particular their Remark 2.1 where they set up the following spectral approach.
We begin by using a formula of Waldspurger/Zhang \cite{Waldspurger} \cite{Z} for $L(f \times \Theta_\chi,1/2)$ to relate the 
first moment to the sum of a fixed automorphic function $F$ evaluated on Heegner points. We then spectrally decompose this 
sum and analyze each part of the spectral contribution separately. The main part of our analysis 
involves the estimation of the contribution of the Maass forms. Roughly, we must estimate an expression 
of the form 
\begin{align*}
\sum_{|t_g| \ll (qD)^{\varepsilon} }\langle F, g \rangle_q W_{D,g}
\end{align*}
where $g$ runs over an orthonormal basis of Maass cusp forms for $\Gamma_0(q)$,  $\langle \cdot, \cdot \rangle_q$ is the Petersson 
inner product on $L^2(Y_0(q))$, and $W_{D,g}$ is the ``Weyl sum'' of $g$ evaluated on a Galois orbit of Heegner points of discriminant $-D$ on $X_0(q)$.  
Watson's formula \cite{W} relates $|\langle F, g\rangle_q|^2$ to the triple product $L$-function $L(f \times f \times g, 1/2) = L(\mathrm{sym}^2 f \times g, 1/2) L(g, 1/2)$ 
while the Waldspurger/Zhang formula relates $|W_{D,g}|^2$ to $L(g, 1/2) L(g \times \chi_D, 1/2)$.  Thus one is naturally led to estimating mean values over quite different families of $L$-functions from those occurring in Theorem \ref{momentthm}.  Various applications of H\"{o}lder's inequality are possible here (see the end of Section \ref{section:cuspforms} for further discussion), and in particular one is led to estimating the average of $L(g \times \chi_D, 1/2)$ which has the advantage of being a $GL_2$ family with smaller conductors than the original $L$-functions.  If $q$ is very small compared to $D$ then the best we can can do is apply a hybrid subconvexity estimate of \cite{BlHa}.  However, if $q$ is somewhat large then it is advantageous to use the following which is of independent interest.
\begin{theorem}
\label{thm:firstmoment}
 Let $T > 0$, $1 \leq M \leq T$, and suppose $\chi_D$ is a primitive quadratic Dirichlet character of conductor $|D|$.  Then for $SL_2(\mz)$, we have
\begin{equation}
\label{eq:firstmoment1}
 \sum_{\substack{%g \in \mathcal{B}_1^{(1)} \\ 
T \leq t_g \leq T + M}} \frac{L(g \times \chi_D, \tfrac12)}{L( \mathrm{sym}^2 g, 1)} + \int_{T}^{T+M} \frac{|L(\tfrac12 + it, \chi_D)|^2}{|\zeta(1 + 2it)|^2} dt \ll (TM + \sqrt{|D|})(|D|MT)^{\varepsilon}.
\end{equation}
Under the additional assumption that $q$ is prime with $(D,q) = 1$, we have for any $M \geq 1$
\begin{equation}
\label{eq:firstmomentq}
 \sum_{\substack{%g \in \mathcal{B}_q \\  
|t_g| \leq M}} \frac{L(g \times \chi_D, \tfrac12)}{L(\mathrm{sym}^2 g, 1)}  \ll (q M^2 + \sqrt{|D|})(|D|Mq)^{\varepsilon},
\end{equation}
where the sum is over Maass newforms for $\Gamma_0(q)$.
In \eqref{eq:firstmoment1} and \eqref{eq:firstmomentq}, the implied constant depends on $\varepsilon > 0$ only.
\end{theorem}
We emphasize that $D$ can be any fundamental discriminant in Theorem \ref{thm:firstmoment} and $q$ can be any odd prime coprime to $D$ (in contrast to Theorem \ref{momentthm}).

Estimates for $L(g \times \chi_D, 1/2)$ and $L(1/2 + it, \chi_D)$ have been studied in various aspects.  For instance, with $g$ fixed (alternately, $t$ fixed), Conrey and Iwaniec \cite{CI} have proved the Weyl-type subconvexity bound $L(g \times \chi_D, 1/2) + |L(1/2 + it, \chi_D)|^2 \ll_{g,t} |D|^{1/3 + \varepsilon}$ by bounding the third moment on average over $g$ of level dividing $|D|$.  The nonnegativity of the central values \cite{KatokSarnak} 
\cite{Biro} is crucial for obtaining this bound since otherwise one cannot drop all but one term.  Recently Michel and Venkatesh \cite{MV2} have obtained a general subconvexity result on $GL_2$ in particular valid for any range of $q$, $T$, and $D$.  For Dirichlet $L$-functions, Heath-Brown \cite{H-B1} \cite{H-B2} showed $L(1/2 + it, \chi) \ll (q(1+|t|))^{3/16 + \varepsilon}$; see also \cite{HuxleyWatt} \cite{Watt}.  Blomer and Harcos \cite{BlHa} have obtained a general hybrid result for $GL_2$ with the use of an amplifier.  Our approach to Theorem \ref{thm:firstmoment} has some elements in common with work of Munshi \cite{Munshi}.
On dropping all but ``one term'', we obtain
\begin{corollary}
 Suppose $g$ is a Hecke-Maass cusp form for $SL_2(\mathbb{Z})$.  Then
\begin{equation}
\label{eq:subconvexity}
 L(g \times \chi_D, \tfrac12) + |L(\tfrac12 + iT, \chi_D)|^2 \ll (T + |D|^{1/2}) (T|D|)^{\varepsilon}.
\end{equation}
Similarly, if $g$ is a Hecke-Maass cusp form for $\Gamma_0(q)$ with $q$ prime and $(q,D) = 1$, we have
\begin{equation}
\label{eq:subconvexityq}
 L(g \times \chi_D, \tfrac12 ) \ll (q + |D|^{1/2})(q|D|)^{\varepsilon},
\end{equation}
\end{corollary}

The conductor of the $L$-functions in \eqref{eq:subconvexity} is $(|D|T)^2$ so this is a subconvexity bound for $T^{1+\delta} \leq |D| \leq T^{A}$ for any fixed $\delta > 0$ and $A > 0$.  For the Dirichlet $L$-function, this estimate improves on Heath-Brown's hybrid bound $L(1/2 + it, \chi_D) \ll (|D|t)^{3/16 + \varepsilon}$ 
provided $t^{5/3} \leq |D| \leq t^3$.  The bound \eqref{eq:subconvexity} is strongest when $|D| \asymp T^2$ in which case it gives a Weyl-type bound.  Similarly, \eqref{eq:subconvexityq} is subconvex for $q^{3/2+\delta} \leq |D| \leq q^A$ with fixed $\delta, A > 0$.

Based on our spectral approach to Theorem \ref{momentthm}, it is desirable to estimate the $L^4$-norm of an $L^2$-normalized Maass form $\widetilde{f}$ of large level $q$ and say bounded Laplace eigenvalue (see \cite{BKY} for investigations into the weight aspect).  Recently, Blomer \cite{BlomerL4}  showed a best-possible (up to $q^{\varepsilon}$) estimate on average for such forms, and observed that $\| \widetilde{f} \|_4^4 \ll q^{-1/3 + \varepsilon}$ by interpolating the sup-norm bound of \cite{HT} with the $L^2$-normalization.  As an easy byproduct of our work here, we record the following
\begin{proposition}
\label{prop:L4}
 Suppose $\widetilde{f}$ is a Hecke-Maass newform of prime level $q$ and spectral parameter $t_f$ which is $L^2$-normalized according to the inner product \eqref{ip}.  Then we have
\begin{equation}
\label{eq:L4}
\| \widetilde{f} \|_4^4 = \int_{Y_0(q)} |\widetilde{f}(z)|^4 \frac{dx dy}{y^2} \ll_{t_f} q^{-1/2 + \varepsilon}.
\end{equation}
\end{proposition}
For comparison, the Lindel\"{o}f Hypothesis would give $O(q^{-1+\varepsilon})$ as the bound in \eqref{eq:L4}.

The form of the bound in Theorem \ref{thm:firstmoment} has some elements in common with a result of Kohnen and Sengupta \cite{KohnenSengupta}, who showed
\begin{equation}
 \sum_{g \in B_k} \frac{L(g \times \chi_D, \tfrac12)}{L(\mathrm{sym}^2 g,1)} \ll_{D, \varepsilon} k^{1+\varepsilon},
\end{equation}
where the sum is over even weight $k$ holomorphic Hecke eigenforms for $SL_2(\mathbb{Z})$ and $(-1)^k D > 0$.  However, with some extra work (see Section \ref{section:largeweight} below), their bound can be improved to
\begin{equation}
\label{eq:KS}
 \sum_{g \in B_k} \frac{L(g \times \chi_D, \tfrac12)}{L(\mathrm{sym}^2 g, 1)} \ll_{\varepsilon} (k + \sqrt{|D|}) (k|D|)^{\varepsilon},
\end{equation}
which has the same form as in Theorem \ref{thm:firstmoment}.  Their proof bypasses some of the analytic techniques in this paper and instead relies on the Shimura correspondence and explicit calculations of the Fourier expansion of half-integral weight Poincare series (alternatively, the Petersson formula for half-integral weight).  This approach gives an elegant and direct proof but it is difficult to generalize it to the Maass form setting.

\section{Preliminaries on Maass forms}

Let $q$ be a positive integer which is either 1 or a prime number. 
Let $h_1, h_2: \H \rightarrow \C$ be $\Gamma_0(q)$-invariant functions on the complex upper half-plane $\H$, 
and define the Petersson inner product 
\begin{align}\label{ip}
\langle h_1, h_2 \rangle_q:= \int_{Y_0(q)}h_1(z)\overline{h_2(z)}\frac{dxdy}{y^2},
\end{align}
where $Y_0(q):=\Gamma_0(q)\backslash \mathbb{H}$ is the open modular curve. The spectrum of $L^2(Y_0(q))$ has an orthonormal 
basis consisting of the constant function, Maass forms, and Eisenstein series corresponding 
to the cusps $\a=0, \infty$ of $\Gamma_0(q)$. An orthonormal basis for the cuspidal spectrum of $L^2(Y_0(q))$ is given by 
\begin{align*}
\Cal{B}:=\Cal{B}_q \cup \Cal{B}_1^{(q)} \cup \Cal{B}_1^{*},
\end{align*}
where $\Cal{B}_q$ is an orthonormal basis of Hecke-Maass newforms of weight 0 for $\Gamma_0(q)$, $\Cal{B}_1^{(q)}$ is a basis of Hecke-Maass cusp forms of weight 
0 for $SL_2(\Z)$ which is orthonormal with respect to the inner product (\ref{ip}), and 
\begin{align*}
\Cal{B}_1^{*}:= \{g_q:~g \in \Cal{B}_1\},
\end{align*}
where (see \cite[Proposition 2.6]{ILS})
\begin{align}\label{geqn}
g_q(z):=\left(1-\frac{q\lambda_g^2(q)}{(q+1)^2}\right)^{-1/2}\left(g(qz)-\frac{q^{1/2}\lambda_g(q)}{q+1}g(z)\right),
\end{align}
and $\lambda_g(n)$ is the $n$-th Hecke eigenvalue of $g$. We have 
\begin{align}\label{lambda1}
\lambda_g(q) &=\pm q^{-1/2} \quad \textrm{when} \quad g \in \Cal{B}_q^{(q)}.
%\lambda_g(q) &\ll q^{1/2}, \quad g \in \Cal{B}_1 \cup \Cal{B}_{1}^{*}. 
\end{align}
It is also convenient to use the notation $\mathcal{B}_1^{(1)}$ to denote the same basis as $\mathcal{B}_1^{(q)}$ but rescaled to be orthonormal with respect to the inner product \eqref{ip} with $q=1$.

Let $t_g:=\sqrt{\lambda_g -\frac{1}{4}}$ denote the spectral parameter of $g \in \Cal{B}$ where $\lambda_g$ is the Laplace 
eigenvalue of $g$. Then by Weyl's law we have 
\begin{align}\label{wl1}
\#\{g \in \Cal{B}_q:~ |t_g| \leq T\} \ll qT^2 \quad \textrm{and} \quad \#\{g \in \Cal{B}_1^{(1)}:~t_g \leq T\} \ll T^{2}.
\end{align}

\section{A formula of Waldspurger/Zhang}

%Let $-D < -4$ be an odd fundamental discriminant such that $(q,D)=1$ and $q$ splits in $K=\Q(\sqrt{-D})$. Let 
%$\mathrm{CL}_K$ be the ideal class group of $K$, $h(-D)$ be the class number of $K$, and $\widehat{\mathrm{CL}}_K$ be the group of 
%ideal class group characters $\chi: \mathrm{CL}_K \rightarrow \C$ of $K$. Let $\{\tau^{\sigma}:~\sigma \in \mathrm{CL}_K\}$ be 
%a Galois orbit of Heegner points of discriminant $-D$ on the modular curve $X_0(q)$. When $q$ is prime there 
%are two orbits that are permuted by the Fricke involution $w_q$ which acts on weight $0$ forms $f$ (invariant under $\Gamma_0(q)$) by $(w_q f)(z) = f(-1/qz)$; 
%for this, see pp. 235-236 of \cite{GZ}.

Let $f$ be an arithmetically normalized Hecke-Maass newform of weight 0 for $\Gamma_0(q)$ with spectral parameter $t_f$,
and $\tilde{f}$ be the $L^2$-normalized newform $\langle \tilde{f}, \tilde{f} \rangle_q=1$ 
corresponding to $f$. Note that $\tilde{f}=a_{\tilde{f}}(1)f,$ and  
if $q$ is a prime (see \cite[eq. (2.9)]{BlomerL4}),
\begin{align}\label{symm}
\langle f, f \rangle_q=\frac{L(\textrm{sym}^2 f, 1)}{2\cosh(\pi t_f)}\frac{q^2}{q-1}.
\end{align}

%Let $\Theta_\chi$ be the weight one theta series for $\Gamma_0(D)$ corresponding to $\chi$ 
Define the completed Rankin-Selberg $L$-function 
\begin{align*}
\Lambda(f \times \Theta_\chi,s):=L_{\infty}(f \times \Theta_\chi,s)L(f \times \Theta_\chi,s),
\end{align*}
where 
\begin{align*}
L_{\infty}(f \times \Theta_\chi,s)=4(2\pi)^{-2s}\Gamma(s+it_f)\Gamma(s-it_f).
\end{align*}
Let $\tau \in \Lambda_D(q)$ be a Heegner point on $X_0(q)$. 
Then one has the following central value formula due to Waldspurger/Zhang \cite{Waldspurger} \cite{Z} (though see p. 647 of \cite{HMi} for the explicit form)
\begin{align}\label{zf1}
\Lambda(f \times \Theta_\chi,\tfrac{1}{2})= \frac{4 \langle f, f \rangle_q}{\sqrt{D}} \Big| \sum_{\sigma \in \mathrm{CL}_K}\chi(\sigma)\tilde{f}(\tau^\sigma)\Big|^2.
\end{align}
%where $c(D) > 0$ is a constant which takes finitely many different values {\bf Elaborate?} 
Note that for $q$ a prime, we can use (\ref{symm}) combined with $\cosh(\pi t) \Gamma(\frac12 + it) \Gamma(\frac12 -it) = \pi$ to write 
\eqref{zf1} as 
\begin{align}\label{zf}
L(f \times \Theta_\chi, \tfrac{1}{2})=\nu(q)\frac{L(\textrm{sym}^2 f, 1)}{\sqrt{D}}
\Big|\sum_{\sigma \in \mathrm{CL}_K}\chi(\sigma)\tilde{f}(\tau^\sigma)\Big|^2,
\end{align}
where $\nu(q) = q^2/(q-1)$.  For $q=1$, the formula \eqref{zf} holds with $\nu(1) = 2$.
% For $q=1$, the formula becomes
% \begin{align}\label{zf'}
% L(f \times \Theta_\chi, \tfrac{1}{2})=2 \frac{L(\textrm{sym}^2 f, 1)}{\sqrt{D}}
% \Big|\sum_{\sigma \in \mathrm{CL}_K}\chi(\sigma)\tilde{f}(\tau^\sigma)\Big|^2.
% \end{align}

\section{Proof of Theorem \ref{momentthm}}
Here we give an overview of the proof of Theorem \ref{momentthm}.
Using the orthogonality relations for the characters $\widehat{\mathrm{CL}}_K$, we obtain from (\ref{zf}) the average 
formula 
\begin{align}\label{af}
M_f(D) := \sum_{\chi \in \widehat{\mathrm{CL}}_K}L(f \times \Theta_\chi,\tfrac{1}{2})=\nu(q) L(\textrm{sym}^2f,1)\frac{h(-D)}{\sqrt{D}}
\sum_{\sigma \in \mathrm{CL}_K}|\tilde{f}(\tau^\sigma)|^2.
\end{align}
Up to a scaling factor, we view the right hand side as the $L^2$ norm of $\widetilde{f}$ restricted to the Galois orbit of Heegner points.

Note that \eqref{af} is invariant under the choice of Galois orbit (when $q$ is prime there are two such orbits) which is consistent with the fact that $|\widetilde{f}|^2$ is invariant under the Fricke involution, and so are the central value of the $L$-functions.

For notational convenience, set $F(z)=|\tilde{f}(z)|^2$. Then spectrally decomposing $F$ in $L^2(Y_0(q))$ yields 
\begin{align}
F(z)=\frac{\langle F,1 \rangle_q}{\textrm{vol}(Y_0(q))}   + \sum_{g \in \Cal{B}}\langle F, g \rangle_q g(z) + \frac{1}{4\pi} 
\sum_{\a} \int_{-\infty}^{\infty}\langle F, E_{\a}(\cdot, \tfrac{1}{2}+it) \rangle_q E_\a(z, \tfrac{1}{2}+it)dt, 
\end{align}
where 
\begin{align*}
E_{\a}(z,s)=\sum_{\gamma \in \Gamma_\a \backslash \Gamma_0(q)}\textrm{Im}(\gamma z)^{s} %, \quad z \in \H , \quad \textrm{Re}(s) > 1 
\end{align*}
is the real-analytic Eisenstein series corresponding to the cusp $\a$ of $\Gamma_0(q)$.  From (\ref{af}) we obtain
\begin{multline}\label{af2}
M_f(D) = %\sum_{\chi \in \widehat{\mathrm{CL}}_K}L(f \times \Theta_\chi,\tfrac{1}{2})=
 \nu(q) L(\textrm{sym}^2f,1)\frac{h(-D)}{\sqrt{D}}
\Big[
\frac{h(-D)}{\textrm{vol}(Y_0(q))} \langle F,1 \rangle_q  \\
 + \sum_{g \in \Cal{B}}\langle F, g \rangle_q W_{D,g}
 + 
\sum_{\a} \int_{-\infty}^{\infty}\langle F, E_{\a}(\cdot, \tfrac{1}{2}+it) \rangle_q W_{D,\a}(t)\frac{dt}{4\pi} \Big],
\end{multline} 
% \begin{align}\label{af2}
% &\sum_{\chi \in \widehat{\mathrm{CL}}_K}L(f \times \Theta_\chi,\frac{1}{2})=\\
% & c(D, t_f)\frac{q^2}{q-1}L(\textrm{sym}^2f,1)\frac{h(-D)^2}{\sqrt{D}}\frac{1}{\textrm{vol}(Y_0(q))} \langle F,1 \rangle_q \notag \\
% & + c(D, t_f)\frac{q^2}{q-1}L(\textrm{sym}^2f,1)\frac{h(-D)}{\sqrt{D}} \sum_{g \in \Cal{B}}\langle F, g \rangle_q W_{D,g} \notag \\
% & + c(D, t_f)\frac{q^2}{q-1}L(\textrm{sym}^2f,1)\frac{h(-D)}{\sqrt{D}} 
% \sum_{\a} \int_{\R}\langle F, E_{\a}(\cdot, \frac{1}{2}+it) \rangle_q W_{D,\a}(t)\frac{dt}{4\pi}, \notag
% \end{align} 
where the hyperbolic Weyl sums are defined by
\begin{align*}
W_{D,g}:=\sum_{\sigma \in \mathrm{CL}_K}g(\tau^{\sigma}), \qquad
% \end{align*}
% and 
% \begin{align*}
W_{D,\a}(t):=\sum_{\sigma \in \mathrm{CL}_K}E_{\a}(\tau^{\sigma}, \tfrac{1}{2}+it).
\end{align*}

Using $\textrm{vol}(Y_0(q)) = \frac{\pi (q+1)}{3}$, and $\langle F, 1 \rangle_q = 1$, 
we find that the constant eigenfunction in (\ref{af2}) gives the main term appearing in \eqref{e8.4}.  
% \begin{align}\label{e8.1}
% \frac{3}{\pi}c(D,t_f)\frac{h(-D)^2}{\sqrt{D}}\frac{q^2}{q^2-1}L(\textrm{sym}^2f,1).
% \end{align}
Proposition \ref{prop:eisenstein} bounds the continuous spectrum which leads to a contribution to $M_f(D)$ of size $O(D^{5/12} (qD)^{\varepsilon} )$.  The contribution from $g \in \mathcal{B}_1^{(q)} \cup \mathcal{B}_1^*$ give a bound of the same size as for the Eisenstein series, by Lemma \ref{lemma:oldforms}.  By Lemma \ref{lemma:Bqestimate}, the contribution to $M_f(D)$ from $\mathcal{B}_q$ gives 
\begin{equation}
 \ll (qD)^{\varepsilon} \min(q D^{7/16}, q^{3/4} D^{1/4} + q^{1/4} D^{1/2}),
\end{equation}
as claimed in Theorem \ref{momentthm}.

% We remark tangentially that in case $D$ is very small compared to $q$ then one can estimate \eqref{af} by the number of terms in the sum times the supremum bound on $\widetilde{f}(z)$.  Harcos and Templier \cite{HT} have shown the bound $\widetilde{f}(z) \ll_{t_f} q^{-1/6 + \varepsilon}$ which leads to $M_f(D) \ll q^{2/3} D^{1/2} (qD)^{\varepsilon}$.  This is stronger than Theorem \ref{momentthm} for $q \gg D^{3+\varepsilon}$.

\section{Period integral formulas}
In this section we will evaluate the magnitude of the Weyl sums $W_{D,g}$, the inner products $\langle F, g\rangle_q$, and the analogous quantities with the continuous spectrum.  First suppose $g \in \mathcal{B}_q$.  Applying (\ref{zf}) with $\chi=\chi_0$ the trivial class group character, we have 
\begin{align}
\abs{W_{D,g}}^2=\frac{\sqrt{D}L(g \times \chi_D, \tfrac{1}{2})L(g,\tfrac{1}{2})}{\nu(q) L(\mathrm{sym}^2 g,1)}  .
\end{align}
It turns out that a similar formula holds for $g \in \mathcal{B}_1^{(q)}$ and $g_q \in \mathcal{B}_1^*$.  Using the nonnegativity of $L(g,1/2)$ and $L(g \times \chi_D, 1/2)$ we shall deduce the following
\begin{lemma}\label{l1} For $g \in \Cal{B}$ we have 
\begin{align}\label{e4.1}
W_{D,g} = \theta_{g,D} \frac{D^{1/4}}{q^{1/2}} \frac{L(g \times \chi_D, \tfrac{1}{2})^{1/2} L(g,\tfrac{1}{2})^{1/2}}{L(\mathrm{sym}^2 g,1)^{1/2}},
\end{align}
where $\theta_{g,D}$ is some complex number satisfying $|\theta_{g,D}| \leq 10$.  
% Furthermore, if $g \in \mathcal{B}_1^{(q)} \cup \mathcal{B}_1^*$, we have
% \begin{equation}
% |W_{D,g}|^2 \ll_{\varepsilon}  (1+\abs{t_g})^{A} q^{-1} D^{\frac{5}{6}+\varepsilon}L(g,\tfrac{1}{2}),
% \end{equation}
% for some fixed $A > 0$.
\end{lemma}
Here we abused notation slightly; if $g_q \in \mathcal{B}_1^*$ then $W_{D, g_q}$ is given by \eqref{e4.1} with $g$ (not $g_q$) appearing on the right hand side.

\begin{proof} If $g \in \Cal{B}_1^{(q)}$ with $\langle g, g \rangle_q=1$, let $\tilde{g}$ be the form associated 
to $g$ in $\mathcal{B}_1^{(1)}$, so that $\langle \tilde{g}, \tilde{g} \rangle_1=1$. Then $g=c\tilde{g}$ where $c^2=1/(q+1)$. 
Following the proof of Harcos and Michel \cite[Theorem 6]{HMi}, we find that
\begin{align*}
\abs{W_{D,g}}^2&=c^2\Big|\sum_{\sigma \in \mathrm{CL}_K}\tilde{g}(\tau^{\sigma})\Big|^2,
\end{align*} 
where $\{\tau^{\sigma}:~\sigma \in \mathrm{CL}_K\}$ is the set of Heegner points of discriminant $-D$ on the 
\textit{level 1} modular curve $X_0(1)$. Then by (\ref{zf1}) with $\chi=\chi_0$, 
\begin{align*}
\abs{W_{D,g}}^2=
\frac{\sqrt{D}}{2(q+1)} \frac{L(g \times \chi_D,\frac{1}{2})L(g,\frac{1}{2})}{L(\mathrm{sym}^2 g, 1)}.
\end{align*}
%Then (\ref{e4.2}) follows for $g \in \mathcal{B}_1^{(q)}$ by a subconvexity bound of Conrey and Iwaniec \cite{CI} we have $L(g \times \chi_D, \frac12) \ll (1 + |t_g|)^{A} D^{1/3 + \varepsilon}$ for some fixed $A > 0$.
 %from these estimates.
If $g_q \in \Cal{B}_1^{*}$, a similar argument using (\ref{geqn}) and the trivial bound $\lambda_g(q) \ll q^{\frac{1}{2}}$ yields 
(\ref{e4.1}) in this case also.
\end{proof}

By Watson's formula \cite{W}, for $g \in \mathcal{B}_q$, we have
\begin{align}
\big|\langle F, g \rangle_q\big|^2=\Big|\int_{Y_0(q)}|\tilde{f}(z)|^2 g(z)\frac{dxdy}{y^2} \Big|^2=\frac{1}{4q^2}
\frac{\Lambda(f \times f \times g,\frac{1}{2})}{\Lambda(\textrm{sym}^2 f,1)^2 \Lambda(\textrm{sym}^2 g, 1)},
\end{align}
where $\Lambda(f \times f \times g,s)$ is the completed triple product $L$-function (see e.g. \cite[chapter 4]{W}). A calculation with the archimedean place (see Section 4 of \cite{BKY} for a convenient reference) gives
\begin{lemma}  Suppose $g \in \mathcal{B}_q$.  Then
\begin{align}
\label{eq:Triple}
\langle F, g \rangle_q = \theta_{f, g} q^{-1} \frac{L(\mathrm{sym}^2 f \times g,\tfrac{1}{2})^{1/2} L(g, \tfrac12)^{1/2}}{L(\mathrm{sym}^2 f, 1) L(\mathrm{sym}^2 g, 1)^{1/2}},
\end{align}
where $\theta_{f,g}$ is some complex constant which satisfies the bound
\begin{equation}
\label{eq:tftg}
 |\theta_{f, g}| \ll \exp(-\pi r(t_f, t_g)), \qquad r(t_f, t_g) = \begin{cases} 0,  &\text{if } |t_g| \leq 2|t_f|, \\ |t_g - 2t_f|,  &\text{if } |t_g| > 2|t_f|. \end{cases}
\end{equation}
\end{lemma}
In fact a more precise estimate is possible but we are not concerned with the $t_f$-behavior in this paper.
Formula (38) of \cite{Nelson} extends \eqref{eq:Triple} to $g \in \mathcal{B}_1^{(q)}$ 
and using the Fricke involution we see that a similar formula holds with $g_q \in \mathcal{B}_1^*$.
%However, we have trivially that
% \begin{equation}
%  |\langle F, g \rangle_q| \leq \sup_{z \in \mathbb{H}} |g(z)| \ll 
% \end{equation}

Next we consider the Eisenstein series.  By unfolding (see Section 5 of \cite{BlomerL4}) we have
\begin{align}
\label{eq:innerproductEisenstein}
\langle F, E_{\mathfrak{a}}(\cdot, s) \rangle_q = |a_{\tilde{f}}(1)|^2\frac{2L(f \times f, s)}{\zeta(2s)}
\frac{\Gamma_{\R}(s)\Gamma_{\R}(s+2it_f)\Gamma_{\R}(s-2it_f)}{2^{s+2}\Gamma_{\R}(s+1)}.
\end{align}
Since $F$ is invariant under the Fricke involution that switches the cusps $0$ and $\infty$, both cusps give the same inner product in \eqref{eq:innerproductEisenstein} (see Section 13.2 of \cite{IwTopics}).
By (\ref{symm}) we conclude 

\begin{align}\label{e7.1}
\langle F, E_{\mathfrak{a}}(\cdot, \tfrac{1}{2}+it) \rangle_q = \theta_{f,t} q^{-1} \frac{|\zeta(\thalf + it)| |L(\mathrm{sym}^2 f, \tfrac12 + it)|}{L(\mathrm{sym}^2 f, 1) |\zeta(1+2it)|},
\end{align}
where $\theta_{f,t}$ satisfies \eqref{eq:tftg} with $t_g$ replaced by $t$.

By (10.30) of \cite{DFI} which generalizes to level $q$ a formula of Gross and Zagier \cite[p. 248]{GZ}, 
\begin{equation}
\label{eq:WeylEisenstein}
\sum_{\mathfrak{a}} W_{D, \mathfrak{a}}(t) = \theta_{t, D} \frac{D^{1/4}}{q^{1/2}} \frac{|\zeta(\tfrac{1}{2}+it)| |L( \tfrac{1}{2}+it, \chi_D)| }{|\zeta(1+2it)|},
\end{equation}
where $\theta_{t,D}$ is some complex number satisfying $|\theta_{t,D}| \leq 10$.  To be precise, we should remark that formula of \cite{DFI} is in a different form than \eqref{eq:WeylEisenstein}; they had an individual cusp on the left hand side, and the sum was over both Galois orbits of Heegner points.  To derive \eqref{eq:WeylEisenstein} we use the fact that the Fricke involution switches the two Galois orbits, and also switches the two Eisenstein series, and so \eqref{eq:WeylEisenstein} follows.

At this point it is easy to establish the following estimate.
\begin{proposition} 
\label{prop:eisenstein}
We have 
\begin{align}\label{e7.3}
\sum_{\a} \int_{-\infty}^{\infty}\langle F, E_{\a}(\cdot, \frac{1}{2}+it) \rangle_q W_{D,\a}(t)\frac{dt}{4\pi} \ll_{t_f, \varepsilon} q^{-1+\varepsilon} D^{\frac{5}{12}+\varepsilon}.
\end{align}
\end{proposition}

\begin{proof}
By the rapid decay of $\theta_{f,t}$ with $|t| > (qD)^{\varepsilon}$, we may truncate the $t$-integral at this point with an acceptable error term.
By \eqref{e7.1}, \eqref{eq:WeylEisenstein}, the convexity bound $L(\mathrm{sym}^2 f, 1/2 + it) \ll q^{1/2 +\varepsilon}$, the Conrey and Iwaniec \cite{CI} bound $L(1/2 + it, \chi_D) \ll D^{1/6} (qD)^{\varepsilon}$  (for $t$ small), and standard \cite{T} lower bounds on $|\zeta(1 + 2it)|$, we immediately obtain \eqref{e7.3}.
\end{proof}

\section{Contribution of the spectrum $\mathcal{B}$}
\label{section:cuspforms}
 In this section we analyze the cusp form sum in \eqref{af2}.  The first step is to finitize the sum over $g$.  
Using self-adjointness of the hyperbolic Laplacian $\Delta=-y^2(\partial_x^2+\partial_y^2)$, Stokes' theorem, 
the Cauchy-Schwarz inequality, and the calculations in Section 5 of \cite{BlomerL4},  we have
\begin{align}\label{Lineq}
\langle F, g \rangle_q = (\frac{1}{4}+t_g^2)^{-A} \langle \Delta^{A} F, g \rangle_q \leq (\frac{1}{4}+t_g^2)^{-A}\norm{\Delta^A F}_2 \norm{g}_{2}
& \ll_{t_f,A} q^{1+\varepsilon} (1+\abs{t_g})^{-A} 
\end{align}
for each integer $A \geq 0$.  Alternatively, one could apply Watson's formula \cite{W} (for $g$ a newform) or Nelson's extension \cite{Nelson} for $g \in \mathcal{B}_1^{(q)} \cup \mathcal{B}_1^*$ to obtain this type of bound.
Since $\langle F, g\rangle_q$ and $W_{D,g}$ grow polynomially in $q$, $D$, etc. we may impose the truncation $|t_g| \leq (qD)^{\varepsilon}$ with a very small error term.  

First we deal with the oldforms by showing
\begin{lemma}  
\label{lemma:oldforms}
We have
\begin{align}\label{estimate5.3}
\sum_{\substack{g \in \Cal{B}_1^{(q)} \cup \Cal{B}_1^{*}\\ t_g \leq (qD)^{\varepsilon}}}\langle F, g \rangle_q W_{D,g}  \ll q^{-1+\varepsilon}D^{\frac{5}{12}+\varepsilon}.
\end{align}
\end{lemma}
\begin{proof}
First we note that for $g \in \mathcal{B}_1^{(q)}$, we have $\sup_{z \in \mathbb{H}} |g(z)| \ll t_g^{1/2} q^{-1/2}$ where $t_g^{1/2}$ is the trivial bound on the sup norm of a level $1$ Hecke cusp form, and the $q^{-1/2}$ comes from the normalization with respect to $Y_0(q)$.  Similarly, inspection of \eqref{geqn} shows that the same bound holds for $g_q \in \mathcal{B}_1^*$.  Thus for $g \in \mathcal{B}_1^{(q)} \cup \mathcal{B}_1^*$, we have $|\langle F, g \rangle_q | \ll t_g^{1/2} q^{-1/2} \langle F, 1 \rangle_q = t_g^{1/2} q^{-1/2}$.  For the Weyl sum, we note by Lemma 5.1 and the Conrey and Iwaniec bound \cite{CI}, we have for $|t_g| \leq (qD)^{\varepsilon}$ that $W_{D,g} \ll q^{-1/2} D^{5/12} (qD)^{\varepsilon}$.  Using the fact that there are $O((qD)^{2\varepsilon})$ oldforms with $t_g \leq (qD)^{\varepsilon}$, we finish the proof.
\end{proof}
Note that the oldforms contribute the same amount as the Eisenstein series (compare Lemma \ref{lemma:oldforms} with Proposition \ref{prop:eisenstein}).
The main work remains to estimate the newforms.  
\begin{lemma}
\label{lemma:Bqestimate}
 We have
\begin{equation}
\label{eq:Bqestimate}
 \sum_{\substack{g \in \mathcal{B}_q \\ |t_g| \leq (qD)^{\varepsilon}}} \langle F, g \rangle_q W_{D,g} \ll (qD)^{\varepsilon} \min(D^{7/16}, \frac{D^{1/4}}{q^{1/4}} + \frac{D^{1/2}}{q^{3/4}}).
\end{equation}
\end{lemma}
\begin{proof}[Proof of Lemma \ref{lemma:Bqestimate}]
Combining \eqref{eq:Triple} and \eqref{e4.1} we obtain
\begin{multline}
\label{eq:FgW}
 \sum_{\substack{g \in \mathcal{B}_q \\ |t_g| \leq (qD)^{\varepsilon}}} \langle F, g \rangle_q W_{D,g} = \frac{D^{1/4}}{q^{3/2} L(\mathrm{sym}^2 f, 1)} 
 \\
 \times
\sum_{\substack{g \in \mathcal{B}_q \\ |t_g| \leq (qD)^{\varepsilon}}} \frac{\theta'_{g,f,D}}{L(\mathrm{sym}^2 g, 1)} L(g \times \chi_D, \tfrac12)^{\frac12} L(g, \tfrac12) L(\mathrm{sym}^2 f \times g, \tfrac12)^{\frac12}.
\end{multline}
We apply H\"{o}lder's inequality with exponents $2, 4, 4$, respectively, obtaining
\begin{equation}
 \sum_{\substack{g \in \mathcal{B}_q \\ |t_g| \leq (qD)^{\varepsilon}}} \langle F, g \rangle_q W_{D,g} \ll \frac{D^{1/4}}{q^{3/2} L(\mathrm{sym}^2 f, 1)} 
M_1^{1/2} M_2^{1/4} M_3^{1/4},
\end{equation}
where
\begin{equation}
\label{eq:3Ms}
\begin{gathered}
 M_1 = \sum_{\substack{g \in \mathcal{B}_q \\ |t_g| \leq (qD)^{\varepsilon}} } \frac{L(g \times \chi_D, \tfrac12)}{L(\mathrm{sym}^2 g, 1)},
\quad
M_2 = \sum_{\substack{g \in \mathcal{B}_q \\ |t_g| \leq (qD)^{\varepsilon}} } \frac{L(g, \tfrac12)^4}{L(\mathrm{sym}^2 g, 1)},
\\
% \end{equation}
% and
% \begin{equation}
 M_3 = \sum_{\substack{g \in \mathcal{B}_q \\ |t_g| \leq (qD)^{\varepsilon}} } \frac{L(\mathrm{sym}^2 f \times g, \tfrac12)^{2}}{L(\mathrm{sym}^2 g, 1)}.
\end{gathered}
\end{equation}
For $M_2$, it is a standard application of the spectral large sieve inequality that
\begin{equation}
 M_2 \ll q (qD)^{\varepsilon}.
\end{equation}
With similar technology combined with some deep inputs on the automorphy of Rankin-Selberg convolutions, we will show
\begin{proposition}
\label{moment}
We have
\begin{equation}
 M_3 \ll q^{2} (qD)^{\varepsilon}.
\end{equation}
\end{proposition}
For $M_1$, we have two different approaches.  For $q$ small compared to $D$ we simply apply the best-known progress towards Lindel\"{o}f for $L(g \times \chi_D, 1/2)$ and multiply by the number of forms (i.e, $q (qD)^{\varepsilon}$).  Currently the best result is \cite{BlHa} which gives $L(g \times \chi_D, 1/2) \ll q^{1/2} D^{3/8} (qD)^{\varepsilon}$.
For $q$ larger we appeal to Theorem \ref{eq:firstmomentq}.  Taken together, we obtain
\begin{equation}
 M_1 \ll (qD)^{\varepsilon} \min(q^{3/2} D^{3/8}, q + D^{1/2}).
\end{equation}
Taking these estimates for granted, we obtain \eqref{eq:Bqestimate} after a short calculation. 
\end{proof}

Now we discuss an alternate arrangements of H\"{o}lder's inequality which may be of interest.  Applying H\"{o}lder's inequality in \eqref{eq:FgW} with exponents $4, 4, 2$, respectively, we obtain
\begin{equation}
 \sum_{\substack{g \in \mathcal{B}_q \\ |t_g| \leq (qD)^{\varepsilon}}} \langle F, g \rangle_q W_{D,g} \ll \frac{D^{1/4}}{q^{3/2} L(\mathrm{sym}^2 f, 1)} M_1'^{1/4} M_2^{1/4} M_3'^{1/2},
\end{equation}
where
\begin{equation*}
 M_1' = \sum_{\substack{g \in \mathcal{B}_q \\ |t_g| \leq (qD)^{\varepsilon}} } \frac{L(g \times \chi_D, \tfrac12)^2}{L(\mathrm{sym}^2 g, 1)}, \qquad M_3' = \sum_{\substack{g \in \mathcal{B}_q \\ |t_g| \leq (qD)^{\varepsilon}} } \frac{L(\mathrm{sym}^2 f \times g, \tfrac12)}{L(\mathrm{sym}^2 g, 1)}.
\end{equation*}
The large sieve inequality easily shows $M_1' \ll (q + q^{1/2} D) (qD)^{\varepsilon}$ and it seems likely that improvements are possible here using current technology.  One may hope to show $M_3' \ll q (qD)^{\varepsilon}$ as this is a family with roughly $q$ elements with conductors of size approximately $q^4$; the weight aspect analog of this estimate was shown in \cite{BKY}.  Conditional on this bound on $M_3'$, one would obtain
\begin{equation*}
 M_f(D) \ll (q^{1/2} D^{1/4} + q^{3/8} D^{1/2}) (qD)^{\varepsilon},
\end{equation*}
which would imply a subconvexity bound for any range of $q$ and $D$ except when one of $q$ or $D$ is fixed.

\section{Sparse equidistribution}
\label{section:equi}
Here we develop a natural formulation of equidistribution of Heegner points of level $q$ and discriminant $D$ as $q$ is allowed to vary with $D$, say restricted by $q \leq D^{\eta}$ for some fixed $\eta > 0$.  The basic difficulty is that the spaces $Y_0(q)$ are varying.  Here we briefly recall the usual definition of equidistribution of say Heegner points of level $1$.
Suppose that $U(z): \Gamma_0(1) \backslash \mathbb{H} \rightarrow \mathbb{R}$ is a smooth, compactly supported function on $Y_0(1)$.  Then equidistribution means
\begin{equation*}
 \lim_{D \rightarrow \infty} \frac{1}{h(-D)} \sum_{\sigma \in \mathrm{CL}_K}  U(\tau^{\sigma}) = \int_{Y_0(1)} U(z) \frac{3}{\pi} \frac{dx dy}{y^2}. 
\end{equation*}
What we will do is construct a sequence (actually, many possible different sequences) of functions $U_q$ invariant on $\Gamma_0(q)$ with ``isometric'' analytic properties, and measure equidistribution with such functions.  To construct these functions, begin with a function $u: \mathbb{H} \rightarrow \mathbb{R}$ that is smooth with support on a compact set $S$ such that no two points of $S$ are $\Gamma_0(1)$-equivalent (this condition is not strictly necessary).  Then $U(z) = \sum_{\gamma \in \Gamma_0(1)} u(\gamma z)$ (a sum with at most one nonzero term) satisfies the above properties.  Now for each $q$, choose $\omega_q \in \Gamma_0(1) / \Gamma_0(q) $ and define $U_q(z) = \sum_{\gamma \in \Gamma_0(q)} u(\omega_q \gamma  z)$.  In other words, we have $U_q(z) = U(z)$, if $z \in \Gamma_0(q) \omega_q^{-1} \Gamma_0(1) \backslash \mathbb{H}$, and $U_q(z) = 0$ otherwise.  One can picture what is going on by taking $S \subset \mathcal{F}$ where $\mathcal{F}$ is the usual fundamental domain for $\Gamma_0(1) \backslash \mathbb{H}$.  Then for each $q$ choose a fundamental domain $\mathcal{F}_q$ for $\Gamma_0(q) \backslash \mathbb{H}$.  Under the natural projection $\pi: \mathcal{F}_q \rightarrow \mathcal{F}$, the set $S$ pulls back to $q+1$ translates of $S$; the construction of $U_q$ is to choose one of these copies to be the support of $U_q$ (restricted to $\mathcal{F}_q$), and $U_q$ restricted to this copy is identical to $U$.  This construction appears (to the authors) to be a natural way to maintain consistent choices of test function for varying $q$'s.  For example, we have obvious statements like
\begin{equation*}
 \Big( \int_{Y_0(1)} |U(z)|^\rho \frac{dx dy}{y^2} \Big)^{1/\rho} = \Big( \int_{Y_0(q)} |U_q(z)|^\rho \frac{dx dy}{y^2} \Big)^{1/\rho},
\end{equation*}
for any $\rho > 0$ (including $\rho=\infty$).  Furthermore, as $q \rightarrow \infty$, the measure of the support of $U_q$ is shrinking compared to the total measure of $\mathcal{F}_q$ so that we are capturing some notion of sparsity.  Now we {\em define} joint equidistribution to mean
\begin{equation}
\label{eq:equidef}
 \frac{1}{h(-D)} \sum_{\sigma \in \mathrm{CL}_K} U_q(\tau^{\sigma}) = \frac{1}{\vol(Y_0(q))} \int_{Y_0(1)} U(z) \frac{dx dy}{y^2}  + E(q,D),
\end{equation}
where as $D \rightarrow \infty$ with $q=q(D) \leq D^{\eta}$, we require
$ E(q,D) = o(q^{-1})$.  In practice we can only expect \eqref{eq:equidef} to hold true for $q$ small enough compared to $D$.  In particular, by volume considerations we cannot even expect the Heegner points to be dense in $Y_0(q)$ unless $q = o(h(-D))$.

\begin{theorem}
\label{thm:equi2}
 Let $U$ and $U_q$ be defined as above.  Then
\begin{equation}
\label{eq:sparseequi}
\frac{1}{h(-D)} \sum_{\sigma \in \mathrm{CL}_K} U_q(\tau^{\sigma}) = \frac{1}{\vol(Y_0(q))} \int_{Y_0(1)} U(z) \frac{dx dy}{y^2} + O(q^{1/4} D^{-1/16 + \varepsilon}),
\end{equation}
and therefore the Heegner points of level $q$ become equidistributed in $Y_0(q)$ provided $q \leq D^{1/20 - \varepsilon}$.  The implied constant depends on $U$ and $\varepsilon$ but not on the choice of $\omega_q$'s.
\end{theorem}
\begin{proof}
 The proof follows similar lines to Theorem \ref{momentthm}.  We begin with the spectral decomposition
\begin{equation}
\label{eq:Uqspectral}
 U_q(z) = \frac{\langle U_q, 1 \rangle_q}{\vol(Y_0(q))} + \sum_{g \in \mathcal{B}} \langle U_q, g \rangle_q g(z) + \frac{1}{4 \pi } \sum_{\mathfrak{a}} \int_{-\infty}^{\infty} \langle U_q, E_{\mathfrak{a}}(\cdot, \tfrac12 + it) \rangle_q E_{\mathfrak{a}}(z, \tfrac12 + it) dt.
\end{equation}
Note $\langle U_q, 1 \rangle_q = \langle U, 1\rangle_1$.  Then inserting \eqref{eq:Uqspectral} into \eqref{eq:sparseequi}, we obtain
\begin{equation}
  \sum_{\sigma \in \mathrm{CL}_K} U_q(\tau^{\sigma})  =  \frac{h(-D) \langle U, 1 \rangle_1 }{\vol(Y_0(q))} +  \sum_{g \in \mathcal{B}} \langle U_q, g \rangle_q W_{D,g} + \sum_{\mathfrak{a}} \int_{-\infty}^{\infty} \langle U_q, E_{\mathfrak{a}}(\cdot, \tfrac12 + it) W_{D, \mathfrak{a}}(t) \frac{dt}{4 \pi}.
\end{equation}
The analog of \eqref{Lineq} shows that we can truncate the spectral sum (and integral) at $|t_g|, |t| \leq D^{\varepsilon}$, with a very small error term.  Then we apply Cauchy-Schwarz and Bessel's inequality to obtain
\begin{multline}
 \Big| \sum_{\sigma \in \mathrm{CL}_K} U_q(\tau^{\sigma})  -  \frac{h(-D) \langle U, 1 \rangle_1 }{\vol(Y_0(q))} \Big|^2 \leq \langle U_q, U_q \rangle_q  \Big(\sum_{\substack{g \in \mathcal{B} \\ |t_g| \leq D^{\varepsilon}}} |W_{D,g}|^2
+ \sum_{\mathfrak{a}} \int_{-D^{\varepsilon}}^{D^{\varepsilon}} |W_{D,\mathfrak{a}}(t)|^2 dt\Big) 
\\
+ O(D^{-100}).
\end{multline}
Then by Lemma \ref{l1}, \eqref{eq:WeylEisenstein}, the \cite{BlHa} bound of $L(g \times \chi_D, \tfrac12) \ll q^{1/2} D^{3/8} (qD)^{\varepsilon}$, and the standard first moment bound for $L(g, \tfrac12)$, (and easier analogues for the continuous spectrum), we have
\begin{equation}
 \Big| \sum_{\sigma \in \mathrm{CL}_K} U_q(\tau^{\sigma})  -  \frac{h(-D) \langle U, 1 \rangle_1 }{\vol(Y_0(q))} \Big|^2 \ll_U q^{1/2} D^{7/8 + \varepsilon}.
\end{equation}
Then with some simplifications we complete the proof.
\end{proof}

\section{Proof of Proposition \ref{moment}}
In this section we prove Proposition \ref{moment}.  The basic idea is to apply the spectral large sieve inequality.  We begin by collecting some standard facts, starting with the spectral large sieve inequality.
\begin{proposition}\label{ls} Let $\{u_j\}$ be an orthonormal basis of Maass cusp forms for $\Gamma_0(q)$. Let 
$\lambda_j(n)$ be the $n$-th Hecke eigenvalue of $u_j$. Let $T \geq 1$ and $N \geq 1$. Then for any 
complex numbers $\{a_n\}_{n=1}^{N}$, we have 
\begin{align}\label{lsi}
\sum_{t_j \leq T} \frac{1}{L(\mathrm{sym}^2 u_j, 1)} \big|\sum_{n=1}^{N}a_n\lambda_j(n)\big|^2 \ll \left(qT^2 +N\log(N)\right)(qT)^{\varepsilon} \sum_{n \leq N} |a_n|^2.
\end{align} 
% where 
% \begin{align*}
% \norm{a}^2=\sum_{n=1}^{N}\abs{a_n}^2.
% \end{align*}
\end{proposition}

By Gelbart and Jacquet \cite{GJ} the symmetric square lift $\textrm{sym}^2f$ is a self-dual automorphic 
form on $\textrm{GL}_3$ with Fourier coefficients $A(m,k)$ satisfying 
\begin{align*}
A(m,1)=\sum_{ab^2=m}\lambda_f(a^2)  \quad \textrm{when} \quad q \nmid m,
\end{align*}
and 
\begin{align*}
A(m,k)=\sum_{d|(m,k)}\mu(d)A(\frac{m}{d},1)A(1,\frac{k}{d}) \quad \textrm{when} \quad q \nmid mk.
\end{align*}
Xiannan Li \cite{L} showed the following uniform bound
\begin{align}\label{rs}
\sum_{mk^2\leq N}\frac{|A(m,k)|^2}{mk} \ll_{t_f} (qN)^{\varepsilon} .
\end{align}
We have
\begin{align}\label{euler}
L(\mathrm{sym}^2 f \times g,s)&=(1-\lambda_g(q)q^{-s})^{-1}(1-\lambda_g(q)q^{-(s+1)})^{-1}L^{(q)}(\mathrm{sym}^2 f \times g,s)\\
& =: \sum_{n=1}^{\infty}\frac{\lambda_{\textrm{sym}^2 f \times g}(n)}{n^s} \notag, %, \qquad \textrm{Re}(s) > 1 \notag
\end{align}
where
\begin{align*}
L^{(q)}(\mathrm{sym}^2 f \times g,s)=\sum_{(mk,q)=1} \frac{A(m,k)\lambda_{g}(m)}{(mk^2)^s}.
\end{align*}

The conductor of $L(\mathrm{sym}^2 f \times g, 1/2)$ is $q^4$.  Using Lemma 3.4 of \cite{LiYoung} (which is a useful variant on the approximate functional equation), we have
\begin{equation}
\label{eq:M3AFE}
 M_3 \ll q^{\varepsilon} \int_{-\log{q}}^{\log{q}} \sum_{|t_g| \leq (qD)^{\varepsilon}} \frac{1}{L(\mathrm{sym}^2 g, 1)}  \Big| \sum_{n \ll q^{2 + \varepsilon}} \frac{\lambda_{\mathrm{sym}^2 f \times g}(n)}{n^{1/2+it}} W(n) \Big|^2 + O((qD)^{-100}),
\end{equation}
where $W(n)$ is some bounded function depending on $q$ but not on $t_g$.  In fact, \eqref{euler} shows that it suffices to bound the sum over $n$ coprime to $q$; one way to see this is to follow the proof of Lemma 3.4 of \cite{LiYoung} and in their (3.34) factor the $L$-function as in \eqref{euler} and trivially bound the contribution from the prime $q$.

By unraveling the definition of Dirichlet series coefficients, and using Cauchy's inequality, we obtain
\begin{multline}
\label{eq:largesievereduction}
 %\sum_{\substack{g \in \Cal{B}_q\\ t_g \leq q^{\varepsilon}}} 
\Big| \sum_{\substack{(n,q) = 1 \\ n \ll q^{2+\varepsilon}}}
\frac{\lambda_{\mathrm{sym}^2 f \times g}(n)}{n^{1/2+it}} W(n) \Big|^2 = 
%\sum_{\substack{g \in \Cal{B}_q\\ t_g \leq q^{\varepsilon}}} 
\Big|  \sum_{\substack{(mk,q) = 1 \\ mk^2 \ll q^{2+\varepsilon}}} \frac{A(m,k) \lambda_g(m)}{m^{1/2+it} k^{1+2it}} W(mk^2) \Big|^2
\\
\ll q^{\varepsilon} \sum_{k \ll q^{1+\varepsilon}} k^{-1} 
%\sum_{\substack{g \in \Cal{B}_q\\ t_g \leq q^{\varepsilon}}} 
\Big|  \sum_{\substack{(m,q) = 1 \\ m \ll  q^{2+\varepsilon}/k^2}} \frac{A(m,k) \lambda_g(m)}{m^{1/2+it}} W(mk^2) \Big|^2.
\end{multline}
Inserting \eqref{eq:largesievereduction} into \eqref{eq:M3AFE} (after freely imposing the condition $(n,q) = 1$), and using Proposition \ref{ls}, we obtain
\begin{equation}
 M_3 \ll q^{\varepsilon} \sum_{k \ll q^{1+\varepsilon}} k^{-1} (q + \frac{q^2}{k^2}) \sum_{m \leq q^{2+\varepsilon}/k^2} \frac{|A(m,k)|^2}{m}.
\end{equation}
Then using \eqref{rs} completes the proof of Proposition \ref{moment}.

\section{Proof of Proposition \ref{prop:L4}}
Next we give the proof of Proposition \ref{prop:L4} which with our current notation gives an upper bound on $\langle F, F \rangle_q$.  By Parseval, we have
\begin{equation*}
 \langle F, F \rangle_q = \sum_{g \in \mathcal{B}} |\langle F, g \rangle_q|^2 + \dots,
\end{equation*}
where the dots indicate the continuous spectrum as well as the constant eigenfunction (which is easily seen to give $O(q^{-1+\varepsilon})$).  As in the proof of Lemma \ref{lemma:oldforms}, the sup-norm of an oldform is $O(q^{-1/2})$ so that these terms also give $O(q^{-1+\varepsilon})$.  By \eqref{eq:Triple} and Cauchy's inequality, we have
\begin{equation}
 \sum_{g \in \mathcal{B}_q} |\langle F, g \rangle_q|^2 \ll q^{-2} M_3^{1/2} M_4^{1/2} , \qquad M_4 = \sum_{\substack{g \in \mathcal{B}_q \\ |t_g| \leq q^{\varepsilon}}}  \frac{L(g, 1/2)^2}{L(\mathrm{sym}^2 g, 1)},
\end{equation}
where $M_3$ is as in \eqref{eq:3Ms}.  Then by Proposition \ref{moment} and the easier bound $M_4 \ll q^{1+\varepsilon}$, we obtain the bound of $O(q^{-1/2 + \varepsilon})$ for the newforms.  By \eqref{e7.1}, and the convexity bound $L(\mathrm{sym}^2 f, 1/2 + it) \ll q^{1/2 + \varepsilon}$ (for $t \ll q^{\varepsilon})$, we see that the Eisenstein series contribute $O(q^{-1+\varepsilon})$ just like the oldforms.  This completes the proof.

\section{Proof of Theorem \ref{thm:firstmoment}}
%We do not bother to consider the continuous spectrum in the level $q$ case since the Fourier expansion of the level $q$ Eisenstein series is similar enough to that of the level $1$ case that one gains no extra information.

We shall treat both bounds in Theorem \ref{thm:firstmoment} simultaneously as much as possible.  For \eqref{eq:firstmomentq} we cover the set of $|t_g| \leq M$ by the set $|t_g| \leq 1$ and $O(M^{\varepsilon})$ subintervals of the form $T \leq t_g \leq T + M'$ with $1 \leq T \leq M$, and $M' = T^{1-\varepsilon}$.

In the $SL_2(\mathbb{Z})$ case we may assume $T \geq |D|^{\varepsilon}$ since otherwise the convexity bound applied to every term immediately gives \eqref{eq:firstmoment1}.
Let $h(t)$ be a smooth, even, nonnegative function satisfying
\begin{equation}
 \label{eq:h}
\begin{cases}
 h(t) \text{ extends to a holomorphic function on } |\text{Im}(t)| \leq A \\
 h(t) \ll (1+|t|)^{-100}, \\
 h(\pm \frac{(2n+1) i}{2})=0 \text{ for } 0 \leq n \leq A,
\end{cases}
\end{equation}
where $A$ is some large positive parameter depending on $\varepsilon > 0$ desired in Theorem \ref{thm:firstmoment}.  We shall eventually take
\begin{equation}
\label{eq:hchoice}
 h(t) = P(t)[\exp(-(\frac{t-T}{M})^2) + \exp(-(\frac{t+T}{M})^2)],
\end{equation}
where $P(t)$ is an even polynomial vanishing at $i/2$, $3i/2$, \dots.  We also use \eqref{eq:hchoice} for $T=0$, $M=1$ to handle $|t| \leq 1$.  For instance, we may take 
\begin{equation}
 P(t) = c \frac{(t^2 + \frac14 )}{T^2 + \frac14} \frac{(t^2 + \frac34)}{T^2+ \frac34} \dots,
\end{equation}
where $c$ is a constant independent of $t$. 
We furthermore suppose $h(t) \geq 1$ for $t$ in the region of interest, i.e., $T \leq t \leq T+M$ or $-\frac14 \leq it \leq \frac14$.  Then in \eqref{eq:firstmoment1} or \eqref{eq:firstmomentq} we can attach the smooth weight $h$ and extend the sum (or integral) to all $t_g$ (or $t$), for purposes of obtaining an upper bound.  Then define for $SL_2(\mz)$,
\begin{equation}
 \mathcal{M}_1 = \sum_{g \in \mathcal{B}_1^{(1)}
} h(t_g) \frac{L(g \times \chi_D, \tfrac12)}{L(\mathrm{sym}^2 g, 1)} + \frac{1}{\pi} \intR h(t) \frac{|L( \tfrac{1}{2}+it, \chi_D)|^2}{|\zeta(1+2it)|^2} dt,
\end{equation}
and for the level $q$ case,
\begin{equation}
 \mathcal{M}_q = \sum_{g \in \mathcal{B}_q} h(t_g) \frac{L(g \times \chi_D, \tfrac12)}{L(\mathrm{sym}^2 g, 1)}.
\end{equation}

Our next step is to use a ``long'' one-piece approximate functional equation for the $L$-functions under consideration.
\begin{proposition}
 Fix an integer $d \geq 1$.  There exists a smooth function $V(x)$ such that for any $L$-function $L(f,s) = \sum_{n=1}^{\infty} \lambda_f(n) n^{-s}$ (as in Chapter 5 \cite{IK}) of degree $d$ and with analytic conductor $ \leq Q$, we have
\begin{equation}
 \label{eq:AFE}
L(f, \tfrac12) = \sum_{n=1}^{\infty} \frac{\lambda_f(n)}{\sqrt{n}} V(n/X) + O_{A,d} (X/Q)^{-A}),
\end{equation}
 where $A > 0$ is arbitrary and the implied constant depends only on $A$ and $d$.
\end{proposition}
\begin{proof}(Sketch)  Let
\begin{equation}
 V(x) = \frac{1}{2 \pi i} \int_{(1)} \frac{\Gamma(2d(s+A+1))}{\Gamma(2d(A+1))} x^{-s} \frac{ds}{s}.
\end{equation}
Then 
\begin{equation}
 \sum_{n=1}^{\infty} \frac{\lambda_f(n)}{\sqrt{n}} V(n/X) = \frac{1}{2 \pi i} \int_{(1)} X^s \frac{\Gamma(2d(s+A+1))}{\Gamma(2d(A+1))} L(f, 1/2 + s) \frac{ds}{s}.
\end{equation}
Shifting contours to $\text{Re}(s) = -A$ picks up the value $L(f, 1/2)$ from the pole at $s=0$, and using the functional equation and (5.114) of \cite{IK}, we obtain the desired estimate by a trivial bound on the new contour of integration.
\end{proof}
The conductor of $L(g \times \chi_D, 1/2)$ is $q D^2 (1+|t_g|^2)$, so we set 
\begin{equation}
Q = q D^2 (T+1)^2, \quad \text{and} \quad X = Q^{1+\varepsilon}.
\end{equation}
Then we have
\begin{equation*}
 \mathcal{M}_1 = \sum_{n=1}^{\infty} \frac{\chi_D(n)}{\sqrt{n}} V(n/X) \Big( \sum_{g \in \mathcal{B}_1^{(1)}
}  \frac{h(t_g)}{L(\mathrm{sym}^2 g, 1)} \lambda_g(n) + \frac{1}{\pi} \intR \frac{h(t)}{|\zeta(1+2it)|^2} \tau_{it}(n) dt \Big) + O(Q^{-100}),
\end{equation*}
and similarly for $\mathcal{M}_q$.  It is very convenient that our method allows us to use this one-piece approximate functional equation because then we do not need to split the family into pieces depending on the parity of $g$, $\chi_D$, etc.

Recall that for $g \in \mathcal{B}_q \cup \mathcal{B}_1^{(q)}$, 
\begin{equation*}
 \frac{1}{L(\mathrm{sym}^2 g, 1)} \asymp q \frac{|\rho_g(1)|^2}{\cosh(\pi t_g)},
\end{equation*}
where $\rho_g(n)$ is the $n$-th Fourier coefficient of $g$ when it is $L^2$-normalized on $Y_0(q)$; see (2.9) of \cite{BlomerL4}.  Thus we have for $l=1$ or $q$, with the notation $\mathcal{B}_q^{(q)}$ denoting $\mathcal{B}_q$,
\begin{equation}
\label{eq:MlPreKuz}
 \mathcal{M}_l \ll l \sum_{n=1}^{\infty} \frac{\chi_D(n)}{\sqrt{n}} V(n/X) \Big( \sum_{g \in \mathcal{B}_l^{(l)}} \frac{ h(t_g)}{\cosh(\pi t_g)}  \rho_g(n) \overline{\rho_g(1)} + (\text{continuous}) \Big) + O(Q^{-100})
\end{equation}
The continuous spectrum contribution to $\mathcal{M}_q$ is nonnegative.
Next, for the case $l=q$, we wish to extend the spectral sum to a full orthonormal basis for $L^2(Y_0(q))$, not just newforms; that is, we include $\mathcal{B}_1^{(q)}$ and $\mathcal{B}_1^*$. We claim that this inclusion only increases the right hand side of \eqref{eq:MlPreKuz}, up to a negligible error term.  It is natural to combine the forms $g \in \mathcal{B}_1^{(q)}$ and the corresponding form $g_q \in \mathcal{B}_1^*$.  
%For $\mathcal{B}_1$, we reverse the above steps to see that one simply obtains $L(g \times \chi_D, 1/2)$ where now $g$ is of level $1$, plus an error term absorbed by that of \eqref{eq:MlPreKuz}.  
Recall from \eqref{geqn} that $g_q(z) = a g(z) + b g(qz)$ for $g\in \mathcal{B}_1^{(q)}$, where $|a| \asymp \frac{|\lambda_g(q)|}{\sqrt{q}}$, and $|b| \asymp 1$.  Thus $\rho_{g_q}(n) = a \rho_g(n) + b \rho_g(n/q)$ (the latter term denoting zero if $q \nmid n$), so that
\begin{equation*}
 \sum_{n=1}^{\infty} \frac{\rho_g(n) \overline{\rho_g(1)} \chi_D(n)}{\sqrt{n}} V(n/X) 
+ \sum_{n=1}^{\infty} \frac{\rho_{g_q}(n) \overline{\rho_{g_q}(1)} \chi_D(n)}{\sqrt{n}} V(n/X)
\end{equation*}
simplifies as
\begin{equation*}
|\rho_g(1)|^2 (1 + a^2) \sum_{n=1}^{\infty} \frac{\lambda_g(n) \chi_D(n)}{\sqrt{n}} V(n/X) + |\rho_g(1)|^2 \frac{ab \chi_D(q)}{\sqrt{q}}   \sum_{n=1}^{\infty} \frac{\lambda_g(n) \chi_D(n)}{\sqrt{n}} V(\frac{n}{X/q}). 
\end{equation*}
Using \eqref{eq:AFE} again (in reverse), we have that this becomes
\begin{equation*}
 |\rho_g(1)|^2(1 + a^2 + \frac{ab \chi_D(q)}{\sqrt{q}} ) L(g \times \chi_D, 1/2) + O(Q^{-100}),
\end{equation*}
where we use the fact that $X/q$ is still larger than the conductor of $L(g \times \chi_D, 1/2)$ since $g$ is level $1$.
 Standard bounds on $\lambda_g(q)$ show that $(1 + a^2 \pm \frac{ab}{\sqrt{q}}) \gg 1$ with an absolute implied constant, so by positivity of the central values again we see that the claim is proved.

Next we require the Kuznetsov formula.
\begin{lemma}  We have
\begin{multline}
 \sum_{g} \overline{\rho_g}(m) \rho_g(n) \frac{h(t_g)}{\cosh(\pi t_g)} + \sum_{\mathfrak{a}} \frac{1}{4 \pi} \intR \tau_{\mathfrak{a}}(m,t) \tau_{\mathfrak{a}}(n,t) \frac{h(t)}{\cosh( \pi t)} dt
\\
= \delta_{m,n} H_0 + \sum_{c \equiv 0 \shortmod{q}} \frac{S(m,n;c)}{c} H\big(\frac{4 \pi \sqrt{mn}}{c} \big),
\end{multline}
where the sum over $g$ runs over an orthonormal basis of Maass forms for $\Gamma_0(q)$, 
\begin{equation}
 H_0 = \pi^{-2} \intR r h(r) \tanh(\pi r) dr
\end{equation}
and
\begin{equation}
\label{eq:gdef}
 H(y) = \frac{2i}{\pi} \intR J_{2ir}(y) \frac{r h(r)}{\cosh( \pi r)} dr.
\end{equation}
\end{lemma}

We then have for $l=1, q$,
\begin{equation}
 M_l \ll l H_0 +  l S_l + O(Q^{-100}), \qquad S_l = \sum_{n=1}^{\infty} \frac{\chi_D(n)}{\sqrt{n}} V(n/X) \sum_{c \equiv 0 \shortmod{q}} \frac{S(n,1;c)}{c} H\big(\frac{4 \pi \sqrt{n}}{c} \big).
\end{equation}
An easy calculation shows $H_0 \ll (T+1)M$, consistent with Theorem \ref{thm:firstmoment} (after adding up $\ll X^{\varepsilon}$ such intervals).

The proof of Theorem \ref{thm:firstmoment} then reduces to showing the following
\begin{proposition}
\label{prop:Kloostermanbound}
 We have
\begin{equation}
 S_l \ll l^{-1} \sqrt{|D|} X^{\varepsilon}.
\end{equation}
\end{proposition}
The proof of Proposition \ref{prop:Kloostermanbound} requires some auxiliary lemmas.
We presently develop some properties of $H(y)$.
\begin{lemma}
Suppose $h$ satisfies \eqref{eq:h}.  Then
\begin{equation}
\label{eq:Hbound}
H^{(j)}(y) \ll_{j,A} (T+1)M \big(\frac{y}{T+1})^{A}, \qquad j=0,1,2,\dots, 200.
\end{equation}
Furthermore, if $h$ is of the form \eqref{eq:hchoice},
there exist functions $u_{\pm}$ satisfying the derivative bound
\begin{equation}
\frac{d^j}{d v^j} u_{\pm}(v) \ll_{j,A} (1 + |v|)^{-A},
\end{equation}
so that
\begin{equation}
\label{eq:Hintegral}
 H(y) = (T+1) \sum_{\delta = \pm} \intR e(\delta T v/M) \cos(y \cosh(\pi v/M)) u_{\delta}(v) dv. 
\end{equation}
\end{lemma}
\begin{proof}
For \eqref{eq:Hbound}, we move the contour of integration in \eqref{eq:gdef} so $\text{Re}(2ir) = A$ and use the bound
\begin{equation}
 J_{A + 2iv}(y) \ll_A \big(\frac{y}{1 + |v|} \big)^{A} \exp(\pi |v|),
\end{equation}
which follows from the integral representation (\cite{GR} 8.411.4)
\begin{equation*}
 J_{\nu}(y) = 2 \frac{(y/2)^{\nu}}{\Gamma(\nu + \half) \Gamma(\half)} \int_0^{\pi/2} (\sin \theta)^{2\nu}  \cos(y \cos \theta) d\theta, \quad \text{Re}(\nu) > -\half.
\end{equation*}
The estimates for the derivatives of $H$ follow similarly by the use of the formula
\begin{equation*}
 \frac{d}{dz} J_{\nu}(z) = \half J_{\nu-1}(z) - \half J_{\nu+1}(z)
\end{equation*}
and changing variables $r \rightarrow r \pm i/2$.  This process shows that $H^{(j)}$ is given by an integral representation similar to that of $H$ but with a slightly different kernel function $r h(r)/\cosh(\pi r)$ replaced by linear combinations of $(r \pm k i/2) h(r \pm ki/2)/\cosh(\pi (r \pm k i/2)$ with $k \in \{ -200, -199, \dots, 200 \}$.  Thus \eqref{eq:Hbound} holds for $1 \leq j \leq 200$ also.

Next we show \eqref{eq:Hintegral}.  Using the fact that $h$ is even and 
8.411.11 of \cite{GR}, that is
\begin{equation*}
 J_{\nu}(y) = \frac{2}{\pi} \int_0^{\infty} \sin(y \cosh t - \frac{\nu \pi}{2}) \cosh(\nu t) dt,
\end{equation*}
we derive the integral representation
 \begin{equation}
 H(y) = \frac{2}{\pi^2} \intR r \tanh( \pi r) h(r) \intR  e(\frac{r v}{\pi}) \cos(y \cosh v) dv dr.
 \end{equation}
 Integrating by parts once in the inner $v$-integral shows that for $V \geq 1$
 \begin{equation*}
 \int_{|v| \geq V} e(\frac{r v}{\pi}) \cos(y \cosh v) dv \ll \frac{1 + |r|}{y \exp(V)}.
 \end{equation*}
 Then by interchanging the orders of integration, we have
\begin{equation}
H(y) = \int_{|v| \leq V} \cos(y \cosh v) U(v) dv + O((T+1)^2 M y^{-1} \exp(-V)),
\end{equation}
where
\begin{equation}
 U(v) = \intR \frac{2}{\pi^2} r \tanh(\pi r) h(r) e(\frac{r v}{\pi}) dr.
\end{equation}
Using the formula \eqref{eq:hchoice}, and changing variables, we have
\begin{equation*}
 U(v) = M (T+1) e(\frac{\delta Tv}{\pi}) \frac{1}{\pi^2} \sum_{\delta = \pm }  \intR \frac{Mr + \delta T}{T+1} P(Mr + \delta T) \tanh(\pi(Mr + \delta T)) e^{-r^2} e(\frac{Mvr}{\pi}) dr. 
\end{equation*}
The integral becomes the Fourier transform, evaluated at $-Mv/\pi$, of a function $W(r) = w_{M,T,\delta}(r)$ satisfying 
\begin{equation}
 \label{eq:wbound}
W^{(j)}(r) \ll_{j,A} (1 + |r|)^{-A}
\end{equation}
and hence its Fourier transform $\widehat{W}$ also satisfies \eqref{eq:wbound}.  Changing variables $v \rightarrow v \pi/M$ and taking the limit as $V \rightarrow \infty$ then gives \eqref{eq:Hintegral}.
\end{proof}

Now we begin the analysis of $S_l$.  By applying a smooth dyadic partition of unity to the $n$-sum, we may assume $N \leq n \leq 2N$, where $1 \leq N \ll X$.
Using \eqref{eq:Hbound} allows us to assume that $c \leq C$, where
\begin{equation}
 C = \frac{ \sqrt{N}}{T+1} X^{\varepsilon}.
\end{equation}
Thus it suffices to show $S(N) \ll l^{-1} \sqrt{|D|} X^{\varepsilon}$, where
\begin{equation}
 S(N) = \sum_{\substack{c \equiv 0 \shortmod{q} \\ c \leq C}} c^{-1} N^{-\half} S_c(N), \qquad S_c(N):= \sum_{n=1}^{\infty} \chi_D(n) S(n,1;c) w(n) H\big(\frac{4 \pi \sqrt{n}}{c} \big),
\end{equation}
where $w$ has support in $[N, 2N]$ and satisfies $w^{(j)}(x) \ll N^{-j}$.

By Poisson summation in $n$ modulo $c|D|$, we have
\begin{equation}
 S_c(N) = \sum_{m \in \mz} \frac{1}{c|D|} a(m;c,D) r(m;c,D),
\end{equation}
where
\begin{equation}
 a(m;c,D) = \sum_{x \shortmod{cD}} \chi_D(x) S(x,1;c) \e{m x}{c|D|},
\end{equation}
and
\begin{equation}
 r(m;c,D) = \int_0^{\infty} w(x) H\big(\frac{4 \pi \sqrt{x}}{c} \big) \e{-mx}{c|D|} dx.
\end{equation}
\begin{lemma}
\label{lemma:arithmetical}
 Suppose that $c = c_1 c_2$ and $D=D_1 D_2$ with $(c_1 D_1, c_2 D_2) = 1$.  Then
\begin{equation}
\label{eq:CRT}
 a(m;c_1 c_2; D_1 D_2) = a(m c_2 \overline{D_2};c_1, D_1) a(m c_1 \overline{D_1};c_2, D_2).
\end{equation}
Furthermore, we have the bound
\begin{equation}
\label{eq:abound}
 |a(m;c,D)| \leq 4^{\nu+2} c |D|^{1/2} (m,c,D)^{1/2},
\end{equation}
where $c = 2^{\nu} c'$ with $(c',2) = 1$.
Finally, $a(0;c,D) = 0$ if $q|c$.
\end{lemma}
\begin{lemma}
\label{lemma:analytical}
 We have
\begin{equation}
\label{eq:rbound}
 r(m;c,D) \ll %TMc \sqrt{N} (1 + \frac{|m| \sqrt{N}}{D})^{-100} (1 + \frac{c TM }{\sqrt{N}})^{-100}.
N X^{\varepsilon} (1 + \frac{|m| \sqrt{N}}{|D|})^{-2}.
\end{equation}
Furthermore, if $T \geq X^{\varepsilon}$ then %$r(m;c,D)$ is $O(X^{-100})$ unless $|m| \asymp \frac{D}{\sqrt{N}}$, and in particular 
$r(0;c,D) = O(X^{-100})$.
\end{lemma}
We postpone the proof of these two lemmas and finish the proof of Proposition \ref{prop:Kloostermanbound}.  In the $\Gamma_0(q)$ case (where the term $m=0$ vanishes), we have
\begin{equation}
\label{eq:SNbound}
S(N) \ll X^{\varepsilon} \sum_{\substack{c \equiv 0 \shortmod{q} \\ c \leq C}} \sum_{m \neq 0} \frac{1}{c \sqrt{N}} |D|^{-1/2} (c,D,m) N (\frac{m \sqrt{N}}{|D|})^{-2} \ll X^{\varepsilon} \sqrt{|D|}/q.
\end{equation}
The case of $\Gamma_0(1)$ gives an identical bound to \eqref{eq:SNbound} (with $q=1$) since in this case the $m=0$ term is practically negligible by Lemma \ref{lemma:analytical}.

\begin{proof}[Proof of Lemma \ref{lemma:arithmetical}]
 By the Chinese remainder theorem, write
\begin{equation*}
 x = x_1 c_2 D_2 \overline{c_2} \overline{D_2} + x_2 c_1 D_1 \overline{c_1} \overline{D_1},
\end{equation*}
where $x_1$ runs modulo $c_1 D_1$, $x_2$ runs modulo $c_2 D_2$, and $c_i \overline{c_i} \equiv D_i \overline{D_i} \equiv 1 \pmod{c_j D_j}$ with $i \neq j$.  Then using $\chi_{D_1 D_2}(x) = \chi_{D_1}(x) \chi_{D_2}(x)$, we have
\begin{equation*}
 a(m;c,D) = \Big( \sum_{x_1 \shortmod{ c_1 D_2}} \chi_{D_1}(x_1) S(x_1 \overline{c_2^2}, 1;c_1) \e{m x_1 \overline{c_2 D_2}}{c_1 D_1} \Big) ( \text{similar} ) ,
\end{equation*}
where the term ``similar'' is identical to the first term but with $c_1$ switched with $c_2$, $D_1$ switched with $D_2$, and $x_1$ switched with $x_2$.  Changing variables $x_1 \rightarrow c_2^2 x_1$ and $x_2 \rightarrow c_1^2 x_2$, we immediately obtain \eqref{eq:CRT}.

We shall bound $a(m;c,D)$ for $c$ and $D$ powers of the same prime $p$; this suffices by \eqref{eq:CRT}.  If $p=2$ we only claim the trivial bound for simplicity.  Now suppose $p$ is odd.
By opening the Kloosterman sum, we have
\begin{equation}
 a(m;p^c, \pm p^D) = \sum_{x \shortmod{p^{c+D}}} \; \sumstar_{y \shortmod{p^c}} \chi_{\pm p^D}(x) \e{xy + \overline{y}}{p^c} \e{mx}{p^{c+D}},
\end{equation}
where $\pm p^D$ is a fundamental discriminant, or $1$.
We consider a variety of cases.
Suppose $D=0$ and $c \geq 1$.  Then the sum over $x$ vanishes unless $y \equiv -m \pmod{p^c}$ in which case necessarily $(m,p) = 1$.  Hence
\begin{equation}
 a(m;p^c, 1) = p^c \e{-\overline{m}}{p^c}.
\end{equation}
If $D = 1$ and $c=0$ then we obtain a Gauss sum, so
\begin{equation}
 a(m;1, \pm p) =  \epsilon_p \leg{m}{p} \sqrt{p}.
\end{equation}
 Suppose now that $c \geq 1$.  Changing variables $x \rightarrow x + p^c$ shows that $a(m;p^c,\pm  p) = \e{m}{p} a(m;p^c, \pm p)$ so it vanishes unless $p|m$.  Accordingly, write $m = p m_1$.  The sum over $x$ is then periodic modulo $p^c$ and is therefore the same sum repeated $p$ times, so we have
\begin{equation}
 a(m;p^c, \pm p) = p \sum_{x \shortmod{p^{c}}} \; \sumstar_{y \shortmod{p^c}} \leg{x}{p} \e{xy + \overline{y}}{p^c} \e{m_1 x}{p^{c}}.
\end{equation}
Suppose that $c=1$.  Then the sum over $x$ is a Gauss sum, and we have
\begin{equation}
 a(m;p, \pm p) = \epsilon_p p^{3/2}  \sumstar_{y \shortmod{p}} \leg{m_1 + y}{p} \e{\overline{y}}{p}.
\end{equation}
By \cite{Schmidt} Theorem 2.6, the inner sum over $y$ is bounded in absolute value by $\sqrt{p}$, so $|a(m;p,\pm p)| \leq p^2 = p^c p^{D/2} (m,p^c, p^D)^{1/2}$, as desired.  For the purpose of proving Proposition \ref{prop:Kloostermanbound} we could get away with using only the trivial bound in place of the Riemann Hypothesis for curves.
 Now suppose that $c \geq 2$.  
%Changing variables $x \rightarrow \overline{y} x$ followed by $y \rightarrow \overline{y}$, we obtain the pleasantly symmetric formula
% \begin{equation}
%  a(m;p^c, p) = p \sum_{x \shortmod{p^{c}}} \; \sumstar_{y \shortmod{p^c}} \leg{xy}{p} \e{x + y + m_1 xy}{p^c}.
% \end{equation}
% If $c$ is even so that $c = 2 \beta$, say, then we write $x = x_1 + p^{\beta} x_2$ and $y = y_1 + p^{\beta} y_2$ where all variables run modulo $p^{\beta}$, we have
% \begin{equation}
%  a(m;p^{2 \beta}, p) = p \sum_{x_1, x_2, y_1, y_2} \leg{x_1}{p} 
% \end{equation}
Write $x = x_1 + p x_2$ where $x_1$ runs modulo $p$, and $x_2$ runs modulo $p^{c-1}$; then the sum over $x_2$ vanishes unless $y \equiv -m_1 \pmod{p^{c-1}}$, so we have
\begin{equation}
 a(m;p^c,\pm p) = p^c \sum_{x_1 \shortmod{p}}  \; \sumstar_{y \equiv - m_1 \shortmod{p^{c-1}}} \leg{x_1}{p} \e{x_1(y+m_1)}{p^c} \e{\overline{y}}{p^c}.
\end{equation}
Now write $y \equiv - m_1 + v p^{c-1}$ where $v$ runs modulo $p$.  The sum over $x_1$ becomes a Gauss sum.  Noting that $\overline{y} \equiv - \overline{m_1} - v \overline{m_1}^2 p^{c-1}$, we have
\begin{equation}
 a(m;p^c,\pm p) = p^{c + \half} \epsilon_p \sum_{v \shortmod{p}} \leg{v}{p} \e{-\overline{m_1} - v\overline{m_1}^2 p^{c-1}}{p^c} = p^{c+1} \e{-\overline{m_1}}{p^c}.
\end{equation}
Having considered all the cases, this completes the proof of \eqref{eq:abound}.

Finally, we show $a(0;c,D) = 0$.  To see this, it suffices to note from the above calculations that $a(0;q^c, 1) = 0$ for $c \geq 1$.  Since $(q,D) = 1$, we have that $a(0;c,D)$ is divisible by $a(0;q^c, 1)$ for some such $c \geq 1$.
%observe that the previous evaluations show that the only way $a(0;p^a, p^b)$ is nonzero is if $a=b=1$ for every prime $p$ dividing $cD$.  This means $c=|D|$, but $q|c$ and $(q,D) =1$, so this is impossible.
\end{proof}

\begin{proof}[Proof of Lemma \ref{lemma:analytical}]
% We begin with some crude bound prior to a more detailed analysis.
%Integration by parts twice shows $r(m;c,D) \ll X^{100} m^{-2}$, say; we may then assume that $m$ is polynomially bounded by $X$.
We begin with the important special case $T \leq X^{\varepsilon}$.  The trivial bound shows $r(m;c,D) \ll (T+1)M N \ll NX^{\varepsilon}$.
For $m \neq 0$, integration by parts twice shows that
\begin{equation}
\label{eq:rparts}
 r(m;c,D) = \frac{-(cD)^2}{(2\pi m)^2} \int_0^{\infty} \frac{d^2}{dx^2} [w(x) H(\frac{4 \pi \sqrt{x}}{c})] e(\frac{-mx}{c|D|}) dx \ll (\frac{|D|}{|m|\sqrt{N}})^2 N (T+1)M X^{\varepsilon},
\end{equation}
which when combined with the trivial bound gives \eqref{eq:rbound}.

Now suppose $T \geq X^{\varepsilon}$ and hence $M \geq X^{\varepsilon}$, adjusting the value of $\varepsilon$ as necessary.  For convenience, note that integrating by parts one more time in \eqref{eq:rparts} shows $r(m;c,D) \ll \leg{|D|}{|m|\sqrt{N}}^3 N(T+1)M X^{\varepsilon}$ which is satisfactory for \eqref{eq:rbound} provided $|m| \gg X^{100}$, say.  Now assume $|m| \ll X^{100}$.
 Inserting \eqref{eq:Hintegral} into the definition of $r(m;c,D)$, writing $2\cos(y) = e^{iy} + e^{-iy}$, and changing variables $x \rightarrow Nx$, we have
 \begin{multline}
 r(m;c,D) = \frac{N(T+1)}{2} \sum_{\delta_1, \delta_2 \in \{ \pm \}}  
\\
\intR \int_0^{\infty} w_N(x) u_{\delta_1}(v) e\big(-\frac{mN}{c|D|}x + \delta_1 \frac{Tv}{M} + 2 \delta_2 \sqrt{x}\frac{\sqrt{N}}{c} \cosh(\frac{\pi v}{M}) \big) dx dv ,
 \end{multline}
 where $w_N(x) = w(Nx)$ so that $w_N^{(j)}(x) \ll_j 1$.
 
 By the rapid decay of $u_{\delta}$, and the fact that $m$ is polynomially bounded by $X$, we may truncate the $v$-integral at $M^{\varepsilon}$ since this gives an acceptable error term.
 Now consider the inner $x$-integral.  With
\begin{equation}
F_1(x) = - \frac{mN}{c|D|} x +\delta_2 2\sqrt{x} \frac{\sqrt{N}}{c} \cosh (\pi v/M), 
\end{equation}
we have
\begin{equation}
F_1'(x) = -\frac{mN}{c|D|} +\delta_2 \frac{\sqrt{N}}{c \sqrt{x}} \cosh (\pi v/M),
\end{equation}
and for $j \geq 2$ we have
\begin{equation}
F_1^{(j)}(x) \asymp \frac{\sqrt{N}}{c} \cosh (\pi v/M) \asymp \frac{\sqrt{N}}{c},
\end{equation}
 for all $x$ in the support of $w_N$.  An easy application of Lemma 8.1 of \cite{BKY} shows that the $x$-integral is very small unless
 \begin{equation}
\label{eq:msize}
 \big| \frac{mN}{cD} \big| \asymp \frac{\sqrt{N}}{c}. % \cosh(\pi v/M).
 \end{equation}
Precisely, if \eqref{eq:msize} does not hold, then
\begin{equation}
 \intR w_N(x) e(F_1(x)) dx \ll_A (\big| \frac{mN}{cD} \big| + \frac{\sqrt{N}}{c} %\cosh(\frac{\pi v}{M})
)^{-A}.
\end{equation}
Since $\sqrt{N}/c \geq X^{\varepsilon}$ (adjusting $\varepsilon$ as necessary), the contribution to $r(m;c,D)$ for parameters where \eqref{eq:msize} does not hold is satisfactory for \eqref{eq:rbound}.  Now assume $\eqref{eq:msize}$ holds.  Before refining the $x$-integral, we turn to the $v$-integral, which takes the form
\begin{equation}
\int_{|v| \leq M^{\varepsilon}} u_{\pm}(v) e(F_2(v)) dv,
\end{equation}
where 
\begin{equation}
 F_2(v) = \delta_1 \frac{Tv}{M} + 2 \delta_2 \sqrt{x}\frac{\sqrt{N}}{c} \cosh(\frac{\pi v}{M}).
 \end{equation}
 We calculate
 \begin{equation}
 F_2'(v) = \delta_1 \frac{T}{M} + 2 \delta_2 \sqrt{x} \frac{\sqrt{N}}{c} \frac{\pi}{M} \sinh( \frac{\pi v}{M}),
 \end{equation}
 and for $j \geq 2$,
 \begin{equation}
 F_2^{(j)}(v) \ll \frac{\sqrt{N}}{c} \frac{1}{M^j}.
 \end{equation}
 Similarly,
 \begin{equation}
 \frac{d^j}{dv^j} w_0(\frac{x}{Z} - \frac{D^2 \cosh^2(\pi v/M)}{m^2 N Z} ) u_{\pm}(v)
 \end{equation}
 If $c \geq \frac{\sqrt{N}}{MT} X^{\varepsilon}$, then $F_2' \gg \frac{T}{M}$ and another easy application of Lemma 8.1 of \cite{BKY} shows that the $v$-integral is very small.  Now suppose
 \begin{equation}
 \label{eq:csharpbound}
 c \leq \frac{\sqrt{N}}{MT} X^{\varepsilon}.
 \end{equation}

Now we return to the $x$-integral.  With an appropriate choice of $\delta_2$, a stationary point exists inside the range of integration at
\begin{equation}
x_0 = \frac{D^2 \cosh^2(\pi v/M)}{m^2 N} \sim \frac{D^2}{m^2 N}. %=: \overline{x_0}.
\end{equation}
For the other choice of sign of $\delta_2$ the first derivative is large and the integral is very small.
The proof of Proposition 8.2 of \cite{BKY} shows that with $Z = Y^{-1/2 + \varepsilon}$, $Y = \sqrt{N}/c$, we have
\begin{equation}
\intR w_N(x) e(F_1(x)) dx = \intR w_N(x) w_0(\frac{x-x_0}{Z}) e(F_1(x)) dx + O(X^{-A}),
\end{equation}
where $w_0$ is any fixed, compactly-supported function that is identically $1$ in a neighborhood of $0$.  The point is that on the complement of $|x-x_0| \leq Z$ the first derivative of $F_1$ is large enough that repeated integration by parts shows that the complementary integral is very small.
Thus we have
\begin{equation}
\label{eq:rinnerv}
|r(m;c,D)| \ll NT \intR  \Big| \int_{|v| \leq M^{\varepsilon}} K(v) e(F_2(v)) dv \Big| dx + O(X^{-100}),
\end{equation}
for some choice of $\delta_i$'s, where
\begin{equation}
 K(v) = w_0(\frac{x}{Z} - \frac{D^2 \cosh^2(\pi v/M)}{m^2 N Z} ) u_{\pm}(v).
 \end{equation}
 Note $K^{(j)}(v) \ll Z^{-1} M^{-(1-\varepsilon)j}$.  
 Proposition 8.2 of \cite{BKY} again shows that the inner $v$-integral in \eqref{eq:rinnerv} is bounded by $M^{1+\varepsilon}/\sqrt{Y}$, giving now
\begin{equation}
r(m;c,D) \ll NT \frac{M}{Y} X^{\varepsilon} = N \frac{c M T}{\sqrt{N}} X^{\varepsilon} \ll N X^{\varepsilon},
\end{equation}
using \eqref{eq:csharpbound}.

Finally, we remark that if $T \gg X^{\varepsilon}$ then the condition \eqref{eq:msize} is incompatible with $m=0$, so $r(0;c,D) = O(X^{-100})$.
\end{proof}

\section{Extending the Kohnen and Sengupta result}
\label{section:largeweight}
Here we quickly sketch the proof of \eqref{eq:KS}.  Formula (8) of \cite{KohnenSengupta} gives that the left hand side is
\begin{equation}
 \ll k^{1+\varepsilon} \Big ( 1+ \big| \sum_{c \geq 1} \frac{1}{c} H_c(|D|, |D|) J_{k-\half} (\pi |D|/c) \big| \Big),
\end{equation}
where $H_c(m,n)$ is a generalized Kloosterman sum corresponding to the half-integral weight multiplier system.  By calculations in Section 3 of \cite{IwHalf}, we have the bound
\begin{equation}
 |H_c(m,n)| \leq 2^{\nu+2} |S_{\chi}(m',n';l)|,
\end{equation}
where $c = 2^{\nu} l$ with $l$ odd, $m' = \overline{2}^{\nu} m$, $n' = \overline{2}^{\nu} n$, and $S_{\chi}(m,n;l)$ is the usual Sali\'{e} sum.  Then by Lemmas 3 and 4 of \cite{IwHalf}, we have $|S_{\chi}(n,n;l)| \leq d(l) l^{1/2} (n,l)^{1/2}$.  Since $k \rightarrow \infty$, we can truncate the sum over $c$ at $c \leq 100 |D|/k$ since otherwise the Bessel function is exponentially small.  We have for $x \gg \nu$ that (see (2.11') of \cite{ILS})
\begin{equation*}
 J_{\nu}(x) \ll \nu^{-1/4} (|x-\nu| + \nu^{1/3})^{-1/4}.
\end{equation*}
Combining these bounds and summing trivially over $c$, we obtain \eqref{eq:KS}, as desired.

On the side, we remark that the general approach used to prove Theorem \ref{thm:firstmoment} can be used to prove \eqref{eq:KS} also.

%%%%%%%%%%%%%%%%%%%%%%%%%%%%%%%%%%%%%%%%%%%%%%%%%%%%%%%%%%%%%%%%%%%%%%%%%%%%%%%%%%%%%%%%%%%%%%%%%%%%%%


\begin{thebibliography}{99}
\bibitem[Bi]{Biro} A. Bir{\'o}, {\it 
Cycle integrals of Maass forms of weight 0 and Fourier coefficients of Maass forms of weight 1/2.}
Acta Arith. 94 (2000), no. 2, 103--152. 

%\bibitem[Bl1]{Bl2} V. Blomer, {\em Non-vanishing of class group $L$-functions at the central point.\/} Ann. Inst. Fourier (Grenoble)
\textbf{54} (2004), 831--847.

\bibitem[Bl]{BlomerL4} V. Blomer, {\em On the 4-norm of an automorphic form\/}, Journal of the European Mathematical Society, to appear. arXiv:1110.4717


\bibitem[BH]{BlHa} V. Blomer and G. Harcos, {\em Hybrid bounds for twisted $L$-functions\/}. J. Reine Angew. Math. \textbf{621} (2008), 53--79.   

%\bibitem[Bu]{Bu} D. A. Burgess, {\em On character sums and L-series.\/} Proc. London Math. Soc. \textbf{12} (1962),
%193--206; II, ibid. \textbf{13} (1963), 524--536.

%\bibitem[BHo]{BlHo} V. Blomer and R. Holowinsky, {\em Bounding sup-norms of cusp forms of large level\/}. Invent. Math. \textbf{179} (2010),  645--681.

\bibitem[BKY]{BKY} V. Blomer, R. Khan, and M. Young, {\it Mass distribution of holomorphic cusp forms.} Preprint, 2012, http://arxiv.org/abs/1203.2573

\bibitem[CI]{CI} B. Conrey and H. Iwaniec, {\em The cubic moment of central values of automorphic L-functions.\/} 
Ann. of Math. \textbf{151} (2000), 1175--1216. 

%\bibitem[DI]{DI} J.-M. Deshouillers and H. Iwaniec, {\em Kloosterman sums and Fourier coefficients of cusp forms.\/} 
%Invent. Math. \textbf{70} (1982/83), 219--288. 
\bibitem[D]{Duke} W. Duke, {\it Hyperbolic distribution problems and half-integral weight Maass forms.}  Invent. Math. 92 (1988), no. 1, 73--90.

\bibitem[DFI]{DFI} W. Duke, J. Friedlander, and H. Iwaniec, {\it Class group $L$-functions.}
Duke Math. J. 79 (1995), no. 1, 1--56. 

\bibitem[FW]{FW} B. Feigon and D. Whitehouse, {\em Averages of central $L$-values of Hilbert modular forms with an application to subconvexity. \/} 
Duke Math. J. \textbf{149} (2009), 347--410.


\bibitem[GJ]{GJ} S. Gelbart and H. Jacquet, {\em A relation between automorphic representations of $GL(2)$ and $GL(3)$.\/}  
Ann. Sci. \'Ecole Norm. Sup. \textbf{11} (1978), 471--542.


\bibitem[GR]{GR} I.S. Gradshteyn, and I.M. Ryzhik, {\it Table of Integrals, Series, and Products}. 
 Translated from the Russian. Sixth edition. Translation edited and with a preface by Alan Jeffrey and Daniel Zwillinger. Academic Press, Inc., San Diego, CA, 2000.

\bibitem[GZ]{GZ} B. Gross and D. Zagier, {\em Heegner points and derivatives
of $L$-series\/}. Invent. Math. \textbf{84} (1986), 225--320.
\bibitem[HaMi]{HMi} G. Harcos and P. Michel, {\em The subconvexity problem
for Rankin-Selberg $L$-functions and equidistribution of Heegner points.
II.\/} Invent. Math.  \textbf{163}  (2006),  581--655.

\bibitem[HT]{HT} G. Harcos and N. Templier, {\em On the sup-norm of Maass cusp forms of large level. III}.  Preprint.

\bibitem[H-B1]{H-B1} D.R. Heath-Brown, {\it Hybrid bounds for Dirichlet L-functions.} Invent. Math. 47 (1978), no. 2, 149--170.
\bibitem[H-B2]{H-B2} D.R. Heath-Brown, {\it Hybrid bounds for Dirichlet L-functions. II.} Quart. J. Math. Oxford Ser. (2) 31 (1980), no. 122, 157--167.

%\bibitem[HL]{HL} J. Hoffstein and P. Lockhart, {\em Coefficients of Maass forms and the Siegel zero (with an appendix by D. Goldfeld,
%J. Hoffstein and D. Lieman)\/}, Ann. of Math. \textbf{140} (1994), 161--181.


\bibitem[HoMu]{HMu} R. Holowinsky and R. Munshi, {\em Level Aspect Subconvexity For Rankin-Selberg $L$-functions\/}. arXiv:1203.1300

\bibitem[HW]{HuxleyWatt} M. N. Huxley and N. Watt, {\it Hybrid bounds for Dirichlet's L-function.} Math. Proc. Cambridge Philos. Soc. 129 (2000), no. 3, 385--415. 


%\bibitem[I]{I} H. Iwaniec, {\em Small eigenvalues of Laplacian for $\Gamma_0(N)$\/}. Acta Arith. \textbf{56} (1990), 65--82.

%\bibitem[I]{I} H. Iwaniec, {\em Spectral methods of automorphic forms.\/} 
%Second edition. Graduate Studies in Mathematics, 53. American Mathematical Society, Providence, RI; 
%Revista Matemitica Iberoamericana, Madrid, 2002. xii+220 pp.


\bibitem[I1]{IwHalf} H. Iwaniec, {\it Fourier coefficients of modular forms of half-integral weight.}
Invent. Math. 87 (1987), no. 2, 385--401. 

\bibitem [I2] {IwTopics} H. Iwaniec, \emph{Topics in Classical Automorphic Forms.} Grad. Stud.
      Math., vol 17, Amer. Math. Soc., 1997.
%\bibitem[I3]{I} H. Iwaniec, {\em Small eigenvalues of Laplacian for $\Gamma_0(N)$\/}. Acta Arith. \textbf{56} (1990), 65--82.

\bibitem[IK]{IK} H. Iwaniec and E. Kowalski, {\em Analytic number theory.\/}
American Mathematical Society Colloquium Publications, 53. American Mathematical
Society, Providence, RI, 2004. xii+615 pp.

\bibitem[ILS]{ILS} H. Iwaniec, W. Luo and P. Sarnak, {\it Low lying zeros of families of $L$-functions.}  Publ. Math. Inst. Hautes \'Etudes Sci.  no. 91, 55-131 (2001).

\bibitem[KaSa]{KatokSarnak} S. Katok and P. Sarnak, {\it Heegner points, cycles and Maass forms.} Israel J. Math. 84 (1993), no. 1-2, 193--227. 

%\bibitem[KiSh]{KS} H. Kim and F. Shahidi, {\em Functorial products for $GL_2 \times GL_3$ and the symmetric cube for $GL_2$.\/} 
%With an appendix by Colin J. Bushnell and Guy Henniart. Ann. of Math. \textbf{155} (2002), 837--893.

\bibitem[KoSe]{KohnenSengupta} W. Kohnen and J. Sengupta, {\it On quadratic character twists of Hecke $L$-functions attached to cusp forms of varying weights at the central point.} Acta Arith. 99 (2001), no. 1, 61--66.

%\bibitem[KZ]{KZ} W. Kohnen and D. Zagier, {\em Values of $L$-series of modular forms at the center of critical strip\/}. Invent. Math. \textbf{64} (1981), 175--198.


\bibitem[KMV]{KMV} E. Kowalski, P. Michel, J. VanderKam, {\em Rankin-Selberg $L$-functions in the level aspect.\/} 
Duke Math. J. \textbf{114} (2002), 123--191.

%\bibitem[LR]{LR} E. Lapid and S. Rallis, {\em On the nonnegativity of $L(\pi,\frac{1}{2})$ for $\textrm{SO}_{2n+1}$\/}. 
%Ann. of Math. \textbf{157} (2003), 891--917. 


\bibitem[L]{L} Xiannan Li, {\em Upper bounds on $L$-functions at the edge of the critical strip\/}. IMRN (2010), 727--755.

\bibitem[LY]{LiYoung} Xiaoqing Li and M. Young, {\em The $L^2$ restriction norm of a $GL_3$ Maass form.} Compositio Math. \textbf{148} (2012), 675--717.  


\bibitem[Mi]{Michel} P. Michel, {\em The subconvexity problem for Rankin-Selberg $L$-functions and equidistribution of Heegner points.\/} 
Ann. of Math. \textbf{160} (2004), 185--236.

\bibitem[MR]{MR} P. Michel and D. Ramakrishnan, {\em Consequences of the Gross/Zagier formulae: Stability of average $L$-values, subconvexity, and non-vanishing mod $p$\/}. 
arXiv:0709.4668 

 \bibitem[MV1]{MVICM} P. Michel and A. Venkatesh, {\it Equidistribution, $L$-functions and ergodic theory: on some problems of Yu. Linnik.} International Congress of Mathematicians. Vol. II, 421--457, Eur. Math. Soc., Z\"{u}rich, 2006. 

\bibitem[MV2]{MV} P. Michel and A. Venkatesh, {\em Heegner points and non-vanishing of Rankin/Selberg $L$-functions.\/} 
Analytic number theory, 169--183, Clay Math. Proc., 7, Amer. Math. Soc., Providence, RI, 2007. 

\bibitem[MV3]{MV2} P. Michel and A. Venkatesh, {\it The subconvexity problem for ${\rm GL}_2$}. Publ. Math. Inst. Hautes \'Etudes Sci. No. 111 (2010), 171--271.

%\bibitem[Mo]{Molteni} G. Molteni, {\em Upper and lower bounds at s=1 for certain Dirichlet series with Euler product.\/}  Duke Math. J. \textbf{111} (2002), 133--158. 
\bibitem[Mu]{Munshi} R. Munshi, {\it On a hybrid bound for twisted L-values.}  Arch. Math. (Basel) 96 (2011), no. 3, 235--245. 
\bibitem[N1]{Nelson} P. Nelson, {\em Equidistribution of cusp forms in the level aspect.} Duke Math. J. 160 (2011), no. 3, 467--501. 
\bibitem[N2]{N} P. Nelson, {\em Stable averages of central values of Rankin-Selberg L-functions: some new variants.\/} arXiv:1202.6313


%\bibitem[R]{R} D. Ramakrishnan, {\em Modularity of the Rankin-Selberg L-series, and multiplicity one for \textrm{SL(2)}.\/} Ann. of Math. \textbf{152} (2000), 45--111. 

\bibitem[S]{Schmidt} W. Schmidt, {\it Equations over finite fields. An elementary approach.} Lecture Notes in Mathematics, Vol. 536. Springer-Verlag, Berlin-New York, 1976. ix+276 pp.

\bibitem[Te]{TemplierDuke} N. Templier, {\it A nonsplit sum of coefficients of modular forms.}  Duke Math. J. 157 (2011), no. 1, 109--165.


\bibitem[Ti]{T} E. C. Titchmarsh, {\em The theory of the Riemann zeta-function.\/} Second edition.  Edited and with a preface by D. R. Heath-Brown. The Clarendon Press, Oxford University Press, New York, 1986. x+412 pp.

\bibitem[Wal]{Waldspurger} J.-L. Waldspurger, {\it Sur les coefficients de Fourier des formes modulaires de poids demi-entier.} J. Math. Pures Appl. (9) 60 (1981), no. 4, 375--484. 

\bibitem[Wats]{W} T. Watson, {\em Rankin triple products and quantum chaos\/}, Annals of Mathematics, to appear. 
\bibitem[Watt]{Watt} N. Watt, {\it On the mean squared modulus of a Dirichlet L-function over a short segment of the critical line.} Acta Arith. 111 (2004), no. 4, 307--403.
\bibitem[Z1]{Z} S. W. Zhang, {\em Gross--Zagier formula for $GL_2$\/}. Asian J. Math. \textbf{5} (2001), 183--290.

\bibitem[Z2]{Z2} S. W. Zhang, {\em Equidistribution of CM-points on quaternion Shimura varieties.\/} Int. Math. Res. Not. 2005, no. 59, 3657--3689.


\end{thebibliography}
\end{document}